\documentclass[12pt]{article}
\usepackage{multirow}
\usepackage[T1]{fontenc}
\usepackage[utf8]{inputenc}
\usepackage{hyphenat}
\usepackage{graphicx}
\usepackage{color}
\usepackage{enumerate}
\usepackage{rotating}
\usepackage{tikz}
\usepackage{amsmath, amsthm, amsfonts, amssymb}

\newtheorem{theorem}{Theorem}[section]

\newtheorem{corollary}[theorem]{Corollary}

\newtheorem{conjecture}{Conjecture}

\newtheorem{observation}[theorem]{Observation}
\newtheorem{claim}{Claim}
\newtheorem{remark}[theorem]{Remark}

\newtheorem{lemma}[theorem]{Lemma}

\newenvironment{prooftw}[2][Proof of Theorem]
{\par\noindent{\bf #1 #2.} }{\hspace*{\fill}\nolinebreak{$\Box$}\bigskip\par}

\usepackage{hyperref}
\usepackage{caption}
\makeatletter
\newcommand\footnoteref[1]{\protected@xdef\@thefnmark{\ref{#1}}\@footnotemark}
\makeatother

\usepackage{algpseudocode}
\usepackage{algorithm}

\begin{document}
\markboth{ }{}
\title{\bf Equitable List Vertex Colourability and Arboricity of Grids}
\date{}
\author{Ewa Drgas-Burchardt\footnote{Faculty of Mathematics, Computer Science and Econometrics,\ University of Zielona G\'ora,\ Prof. Z. Szafrana 4a,\ 65-516 Zielona G\'ora,\ Poland.
\ e-mail: E.Drgas-Burchardt@wmie.uz.zgora.pl; E.Sidorowicz@wmie.uz.zgora.pl},
Janusz Dybizba\'nski\footnote{Institute of Informatics,\ Faculty of Mathematics, Physics and Informatics,\ University of Gda\'nsk,\ 80-308 Gda\'nsk,\ Poland. \ e-mail: jdybiz@inf.ug.edu.pl; hanna@inf.ug.edu.pl},
\\ Hanna Furma\'nczyk\footnotemark[2], El\.zbieta Sidorowicz\footnotemark[1]
}
\markboth{E. Drgas-Burchardt, J. Dybizba\'nski, H. Furma\'nczyk, E. Sidorowicz}{Equitable List Vertex Colourability and Arboricity of Grids}

\maketitle

\begin{abstract}
A graph $G$ is equitably $k$-list arborable if for any $k$-uniform list assignment $L$, there is an equitable $L$-colouring of $G$ whose each colour class induces an acyclic graph. 
The smallest number $k$ admitting such a coloring is named equitable list vertex arboricity and is denoted by $\rho_l^=(G)$. 
Zhang in 2016 posed the conjecture that if $k \geq \lceil (\Delta(G)+1)/2 \rceil$ then $G$ is equitably $k$-list arborable. 
We give some new tools that are helpful in determining values of $k$ for which a general graph is equitably $k$-list arborable. We use them to  prove the Zhang's conjecture for $d$-dimensional grids where $d \in \{2,3,4\}$
and give new bounds on $\rho_l^=(G)$ for 
general 
graphs and for $d$-dimensional grids with $d\geq 5$.
\end{abstract}



\section{Introduction}

All graphs considered in this paper are simple and undirected. For a graph $G$, we use $V(G)$, $E(G)$, and $\Delta(G)$ to denote vertex set, edge set, and the maximum degree of $G$, respectively. By $G[V^\prime]$ we mean  the subgraph of $G$ induced by a vertex subset $V^\prime$. 
To simplify the notation we write $G-V^\prime$  instead of $G[V(G)\setminus V^\prime]$. Analogously, we write $G-E^\prime$ to denote the graph obtained from $G$ by the deletion of an edge subset $E'$. By $G_1\cup G_2$ we mean the union of disjoint graphs $G_1,\;G_2$, i.e. the graph with vertex set $V(G_1)\cup V(G_2)$ and edge set $E(G_1)\cup E(G_2)$.

The symbol ${\mathbb N}$ stands for the set of positive integers, and moreover ${\mathbb N}_0={\mathbb N}\cup \{0\}$. Let $a,b\in {\mathbb N}_0$. If $a<b$ then $[a,b]$ denotes the set  $\{a,a+1,\ldots, b-1,b\}$, if $a=b$ then $[a,b]=\{a\}$, and  if $a>b$ then $[a,b]=\emptyset$. We adopt the convention $[1,b]=[b]$, moreover $[b]_{ODD}$ and $[b]_{EVEN}$ denote the sets of odd integers and even integers in $[b]$, respectively.

A \emph{colouring} of a graph $G$ is a mapping $c:V(G)\rightarrow \mathbb{N}$. A \emph{coloured graph} is then a pair $(G,c)$, where $G$ is a graph and $c$ is its colouring. 
A colouring of a graph $G$ is {\it proper} if each colour class induces an edgeless graph. 
A $k$-\emph{colouring} of a graph $G$ is a mapping $c:V(G)\rightarrow [k]$. A graph $G$ is {\it properly} $k$-{\it colourable} if there is a proper $k$-colouring of $G$. A graph $G$ is $k$-{\it arborable} if there is a $k$-colouring of $G$  such that each colour class induces an acyclic graph.

Let $L$ be a \emph{list assignment} (\emph{for a graph $G$}), i.e. a mapping that assigns to each vertex $v \in V(G)$ a set $L(v)$ of allowable colours. An $L$-\emph{colouring} of $G$ is a  colouring of $G$ such that for every $v \in V(G)$ the colour on $v$ belongs to $L(v)$.  A list assignment $L$ is $k$-\emph{uniform} if $|L(v)|=k$ for all $v \in V(G)$. A graph $G$ is   $k$-{\it choosable} if for each $k$-uniform list assignment $L$, we can find a proper $L$-colouring of $G$.   A graph $G$ is \emph{$k$-list  arborable} if,  given a $k$-uniform list assignment $L$, we can find an $L$-colouring of $G$ so that each colour class induces an acyclic subgraph of $G$. By $\chi(G)$, $\rho(G)$, $ch(G)$,  $\rho_l(G)$ we denote the minimum  $k\in \mathbb{N}$ such that $G$ is:  properly $k$-colourable,  $k$-arborable,  $k$-choosable,  $k$-list arborable, respectively. We call these numbers the \emph{chromatic number} of $G$,  the \emph{vertex arboricity} of $G$, the \emph{choice number} of $G$, the \emph{list vertex arboricity} of $G$, respectively. 
The invariant $\rho(G)$ was first introduced by Beineke in $1964$ \cite{beineke} and then it was investigated by many researchers. For example, Chartrand, Kronk, and Wall in $1968$ \cite{chart} proved that $\rho(G) \leq \lceil (\Delta(G)+1)/2\rceil$ for every graph $G$. Next, in 1995, Borowiecki, Drgas-Burchardt, and Mih\' ok \cite{BoDrMi95} introduced the list version of these problem.
They showed that $\rho_l(G) \leq \lceil (\Delta(G))/2\rceil$ for every connected graph $G$ excluding cycles and 
complete graphs of odd order.

In this paper we are mostly interested in  a non-classical model of graph colouring, known as equitable. 
A $k$-colouring of a graph $G$ is \emph{equitable} when each of its colour classes is of the cardinality either $\lceil |V(G)|/k \rceil$ or $\lfloor |V(G)|/k \rfloor$.  A graph $G$ is \emph{equitably properly $k$-colourable} if there exists an equitable proper $k$-colouring of $G$. The definition was firstly introduced by Meyer \cite{meyer} in 1973.
Recently, Wu, Zhang and Li \cite{wu} introduced the equitable version of vertex arborocity. 
A graph $G$ is \emph{equitably $k$-arborable} if there exists an equitable $k$-colouring of $G$  whose each colour class induces an acyclic graph. 
In the list version, given a $k$-uniform list assignment $L$ for $G$, we call an $L$-colouring of $G$ \emph{equitable} when each colour class has the cardinality at most  $\lceil |V(G)|/k \rceil$ (see \cite{KoPeWe03}). A graph $G$ is \emph{equitably $k$-choosable} when for any $k$-uniform list assignment $L$, there is an equitable proper $L$-colouring of $G$. A graph $G$ is \emph{equitably $k$-list arborable} when for any $k$-uniform list assignment $L$, there is an equitable $L$-colouring of $G$ whose each colour class induces an acyclic graph. The last definition was given by Zhang \cite{zhang} in 2016.
By $\chi^=(G)$,  $\rho^=(G)$, $ch^=(G)$, $\rho_l^=(G)$ we denote the minimum  $k\in \mathbb{N}$ such that $G$ is: equitably properly  $k$-colourable,   equitably $k$-arborable,  equitably $k$-choosable,  equitably $k$-list arborable, respectively. 
The numbers $\chi^=(G)$,  $\rho^=(G)$, $ch^=(G)$, $\rho_l^=(G)$  are called the \emph{equitable chromatic number} of $G$, the \emph{equitable vertex arboricity} of $G$,  the \emph{equitable choice number} of $G$, the \emph{equitable list vertex arboricity} of $G$, respectively. 

Hajn\' al and Szemer\'{e}di (\cite{HaSe70}) proved that a graph $G$ is equitably properly $k$-colourable whenever $k\geq \Delta(G)+1$. It caused a question posed  by P. Erd\"{o}s. Kostochka, Pelsmajer, and West \cite{KoPeWe03} conjectured the list version of this theorem. 

\begin{conjecture}[\cite{KoPeWe03}]
 If $k\in {\mathbb N}$ and  $k \geq \Delta(G)+1$ then every graph $G$ is equitably $k$-choosable. \label{con0}
\end{conjecture}

It has to be mentioned herein that equitable $k$-colouring is not monotone with respect to $k$. It means that there are graphs that are equitably $k$-colourable and not equitably $t$-colourable for some $t < k$. To the best of our knowledge  there are no results of this type on equitable $k$-choosability nor equitable $k$-list arborability.

On the other hand, Zhang \cite{zhang} formulated in $2016$  the following conjectures.

\begin{conjecture}[\cite{zhang}]
 For every graph $G$ it holds $\rho_l^=(G) \leq \lceil (\Delta(G)+1)/2 \rceil$. \label{con1}
\end{conjecture}

\begin{conjecture}[\cite{zhang}]
If $k\in {\mathbb N}$ and  $k \geq \lceil (\Delta(G)+1)/2 \rceil$ then every graph $G$ is equitably $k$-list  arborable. \label{con2}
\end{conjecture}

Zhang \cite{zhang} confirmed above two conjectures for complete graphs, 2-degenerate graphs, 3-degenerate claw-free graphs with maximum degree at least 4, and planar graphs with maximum degree at least $8$. 
Our results confirm above conjectures for some Cartesian products of paths, i.e. for some grids. 

Given two graphs $G_1$ and $G_2$, the \emph{Cartesian product} of $G_1$ and $G_2$, denoted by $G_1 \square G_2$, is defined to be a graph  whose vertex set is
$V(G_1) \times V(G_2)$ and edge set consists of all the edges joining vertices $(x_1,y_1)$ and $(x_2,y_2)$ when either $x_1=x_2$ and $y_1y_2 \in E(G_2)$ or $y_1=y_2$ and $x_1x_2\in E(G_1)$. Note that the Cartesian product is commutative and associtive. Hence the graph $G_1 \square \cdots \square G_d$ is unambiguously defined for any $d \in \mathbb{N}$. Let $P_n$ denote a path on $n$ vertices. 
Notice that when $G=G_1 \square \cdots \square G_d$ and each of the factors $G_i$ of $G$ is $P_2$ then $G$ is a $d$-\emph{dimensional hypercube}. Similarly, when each of the factors $G_i$ is a path on at least two vertices then $G$ is a $d$-\emph{dimensional grid} (cf. Fig.~\ref{p34}). By \emph{grids} we mean the class of all $d$-dimensional grids taken over all $d\in \mathbb{N}$.


\begin{figure}[htb]
\begin{center}
\includegraphics[scale=1]{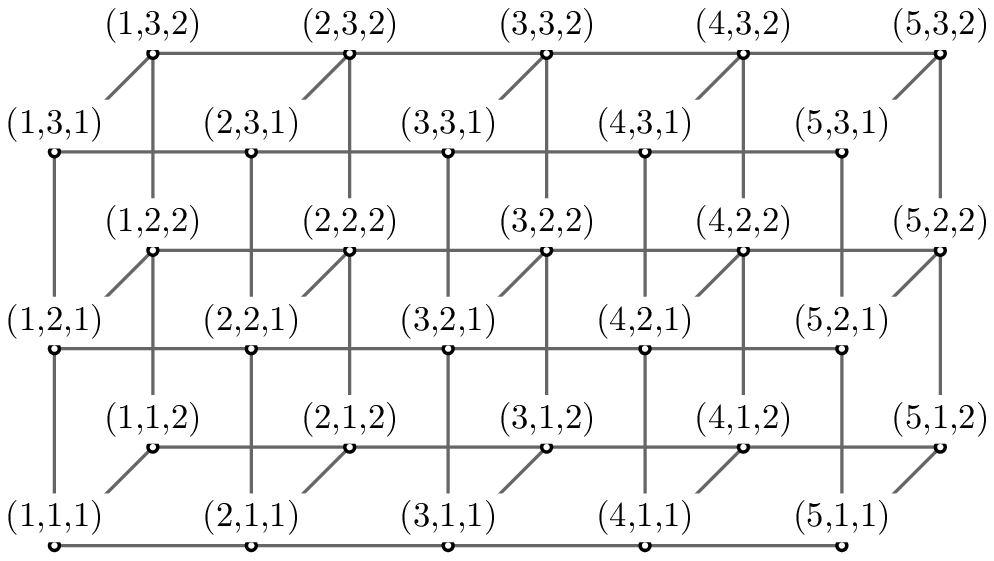}
\caption{3-dimensional grid $P_5 \square P_3 \square P_2$.}
\label{p34}
\end{center}
\end{figure}


Nakprasit and Nakprasit \cite{nakp} proved that the problem of equitable vertex arboricity is NP-hard. Thus the problem of equitable list vertex arboricity cannot be easier.  We are interested in determining polynomially solvable cases.  
We will use the following known lemmas. By $N_G(x)$ we denote \emph{neighborhood} of a vertex $x$ in $G$, i.e. the set of adjacent vertices to $x$.



\begin{lemma}[\cite{KoPeWe03,Pe04}]
\label{lem:Kostochka-Pelsmajer-West_2}
Let $k\in \mathbb{N}$ and $S=\{x_1,\ldots, x_k\}$, where $x_1,\ldots,x_k$ are distinct vertices of $G$. If $G-S$ is equitably $k$-choosable 
and 
\begin{equation}\label{ineq}
|N_G(x_i) \backslash S| \leq i-1
\end{equation}
holds for every $i\in [k]$ then $G$ is equitably  $k$-choosable.
\end{lemma}

\begin{lemma}[\cite{zhang}]\label{lm}
Let $k\in \mathbb{N}$ and $S=\{x_1,\ldots, x_k\}$, where $x_1,\ldots,x_k$ are distinct vertices of $G$. If $G- S$ is equitably $k$-list  arborable 
and 
\begin{equation}
|N_G(x_i) \backslash S| \leq 2i-1 \label{eqlm}
\end{equation}
holds for every $i\in [k]$ then $G$ is equitably $k$-list arborable.
\end{lemma}

In this paper we investigate the problem of equitable list vertex arboricity of graphs. The remainder of the paper is organized as follows.
In Section \ref{lem:generalization} we generalize Lemmas \ref{lem:Kostochka-Pelsmajer-West_2} and \ref{lm} in such a way that their new versions guarantee the continuity of the equitable choosability and equitable list vertex arboricity of graphs. We give also a new tool using the equitable choosability of a subgraph $H$ covering graph $G$ (Lemma \ref{lem:covering}). These tools (Lemmas \ref{lem:vertex_partition}, \ref{lem:vertex_partition_arboricity}, and \ref{lem:covering}) lead to new bounds on $\rho_l^=(G)$, 
for any graph $G$.
Since the new tool uses the notation of equitable choosability we dedicate Section \ref{colouring-grids} to this notation for some graphs related to grids. Finally, we apply all the lemmas to confirm the correctness of Zhang's conjectures for $d$-dimensional grids, $d\in \{2,3,4\}$, 
and to give new bounds on $\rho_l^=(G)$ for $d$-dimensional grids with $d\geq 5$ (Section \ref{sec:arboricity_grids}). We conclude the paper with posing some new conjectures concerning equitable list vertex arboricity of graphs.


\section{Some auxiliary tools and general bounds on $\rho_l^=(G)$}
\label{lem:generalization}

In the literature a lot of proofs of results on  equitable choosability are done by induction on the number of vertices of a graph and by usage of Lemma \ref{lem:Kostochka-Pelsmajer-West_2}. It means, to show that $G$ is equitably $k$-choosable, the set $S\subseteq V(G)$ that fulfills the inequality (\ref{ineq}) is determined and next the induction hypothesis is applied to the graph $G-S$.  Repeated application of this approach defines  a  partition $S_1\cup\cdots \cup S_{\eta +1}$ of $V(G)$ such that the following both conditions hold.  

\begin{itemize}
\item $|S_1|\le k$ and $|S_{j}|=k$ for $j\in[2,\eta+1]$;
\item for each $j\in[2,\eta+1]$ there is an ordering of vertices of $S_j$, say $x_1^j,\ldots,x_k^j$, that fulfills the inequality $|N_G(x_i^j) \cap (S_1\cup\cdots \cup S_{j-1})| \leq i-1 $ for every $i\in[k]$.
\end{itemize}

\noindent In this section we prove that if $G$ has such a partition then $G$ is not only equitably $k$-choosable but also is equitably $t$-choosable for every $t\in {\mathbb N}$ satisfying $t\ge k$. Next, we observe that the similar result for a graph to be equitably $k$-list arborable can be formulated.

Let $k\in \mathbb{N}$. A $k$-{\it partition} of a graph $G$ is a partition of the vertex set of $G$ into $\left\lceil |V(G)|/k\right\rceil $  sets. The $k$-partition is {\it special} if all sets of the $k$-partition, except at most one, have $k$ elements. Let $G$ be a graph and $c$ be its vertex colouring (not necessarily proper). A set $S\subseteq V(G)$ is {\it rainbow} in the coloured graph $(G,c)$ if all vertices in $S$ are coloured differently. A $k$-partition of the coloured graph $(G,c)$ is {\it rainbow} if every set of the $k$-partition is rainbow. It is easy to see the following fact.

\begin{observation}
\label{obs:rainbow-bounded}
Let $k\in {\mathbb N}$ and  $(G,c)$ be a coloured graph. If there is a rainbow $k$-partition of  $(G,c)$ then each colour appears on at most $\left\lceil |V(G)|/k\right\rceil$ vertices of $G$.
\end{observation}

\begin{lemma}
\label{lem:rainbow-colourable}
Let $k\in {\mathbb N}$. A graph $G$ is equitably  $k$-choosable if and only if for every $k$-uniform list assignment $L$ there is a proper $L$-colouring $c$ of $G$ such that $(G,c)$ has a rainbow $k$-partition.
\end{lemma}

\begin{proof}
Obviously, if for every $k$-uniform list assignment $L$ there is a proper $L$-colouring $c$ of $G$ such that $(G,c)$ has a rainbow $k$-partition then each colour class has the cardinality at most $\left\lceil |V(G)|/k\right\rceil$, by Observation \ref{obs:rainbow-bounded}. It means that this $L$-colouring $c$ is equitable, and hence $G$ is equitably $k$-choosable.

To prove the opposite implication, suppose that $G$ is equitably $k$-choosable and $L$ is a $k$-uniform list assignment for $G$. It follows that there is a proper $L$-colouring $c$ of $G$ such that each colour class has at most $\left\lceil |V(G)|/k\right\rceil$ elements. Let $|V(G)|=\eta k+r$, where $\eta\in \mathbb{N}_0, \; r\in [k]$. Thus $\eta+1=\left\lceil |V(G)|/k\right\rceil$, and so  each colour class contains at most $\eta+1$ vertices.
Assume, on the contrary, that there is no rainbow $k$-partition of $(G,c)$. Among all partitions of $(G,c)$ into rainbow sets,  let $V_1\cup \cdots \cup V_{t}$ be one with the smallest $t$. Since there is no rainbow $k$-partition, we have $t>\eta+1$. 
Without loss of generality, we may assume that  $V_1\cup \cdots \cup V_{t}$ is the rainbow partition with $|V_1|\leq \cdots \leq |V_t|$ and with the minimum cardinality of $V_1$. Let  $|V_1|=s$ and $x\in V_1$. Since we have at most $\eta+1$ vertices coloured with $c(x)$ and $t>\eta+1$, there is a set $V_i$ such that $V_i\cup\{x\}$ is rainbow. If $s=1$ then $V_2\cup \cdots\cup (V_i\cup \{x\})\cup\cdots \cup V_{t}$ is the partition with less number of rainbow sets, a contradiction. If $s>1$ then we get the rainbow partition  $V_1\setminus \{x\}\cup \cdots\cup (V_i\cup \{x\})\cup\cdots \cup V_{t}$ that
contradicts with the minimum cardinality of $V_1$.

\end{proof}

\begin{lemma}
\label{lem:rainbow-arborable}
Let $k\in {\mathbb N}$. A graph $G$ is equitably $k$-list arborable if and only if for every $k$-uniform list assignment $L$ there is an $L$-colouring $c$ in which every colour class induces an acyclic graph and such that $(G,c)$ has a rainbow $k$-partition. 
\end{lemma}
\begin{proof}
We repeat all the steps of the proof of Lemma \ref{lem:rainbow-colourable}, but in each case when we refer to the colouring $c$ of a graph $G$ we assume or state that each colour class in $c$ is acyclic instead of the assumption that $c$ is proper. Additionally, we substitute the notion of equitable $k$-choosability by the notion of equitable $k$-list arborability.

\end{proof}


\begin{lemma}
\label{lem:rainbow_partition}
Let $k\in {\mathbb N}$ and $(G,c)$ be a  coloured graph. If there is a rainbow special $k$-partition of  $(G,c)$ then there is also a rainbow special $x$-partition of $(G,c)$ for every integer $x$ such that $x\le k$.
\end{lemma}

\begin{proof}
\noindent Let $|V(G)|=\eta k+r_1$, where $\eta\in \mathbb{N}_0, \; r_1\in [k]$. Let $S_1\cup S_2\cup\cdots \cup S_{\eta +1}$ be a rainbow special $k$-partition of $(G,c)$ such that $|S_1|=r_1$ and $|S_i|=k$ for $i\in[2,\eta+1]$. We show that there is a rainbow special $x$-partition, for every $x \leq k$.

Arrange vertices of $G$ in the list in such a way that:
\begin{itemize}
\item vertices from $S_i$ are placed before vertices from $S_j$ for $i<j$,
\item vertices from $S_1$ are placed in any order at the top of the list,
\item each vertex from $S_i$, for $i>1$, is placed in the list in such a way that its colour is different from the colours of $k-1$ previous vertices in the list or its colour is different from the colours of all previous vertices in the list, if the number of previous vertices is smaller than $k-1$.
\end{itemize}
Since sets $S_i$ are rainbow, for every $i$, then the above described arrangement of vertices is possible. Assume that $(v_1, v_2, \dots, v_{|V(G)|})$ is the list of vertices created in such a way. 
Let $|V(G)|=\beta x+r_2$, where $\beta\in \mathbb{N}_0, \; r_2\in [x]$. 

Sets $R_i = \{ v_{(i-1)\cdot x+1}, \dots, v_{i\cdot x} \}$, for $1 \leq i \leq \beta$ and
$R_{\beta+1}=\{ v_{\beta\cdot x+1}, \dots, v_{|V(G)|} \}$ form an $x$-partition. It is easy to see that this partition is rainbow and special.
\end{proof}

\begin{lemma}
\label{lem:vertex_partition}
Let $k\in {\mathbb N}$. If a graph $G$ has a special $k$-partition $S_1\cup \cdots \cup S_{\eta+1}$ such that $|S_1|\le k$ and $|S_j|=k$ for $j\in[2 , \eta+1]$, moreover, if for every $j\in[2 , \eta+1]$ there is an ordering $x_1^j,\ldots ,x_k^j$ of vertices of the set $S_j$ that for every $i\in[k]$ the inequality
\begin{equation}
|N_G(x_i^j) \cap (S_1\cup\cdots \cup S_{j-1})| \leq i-1,  \label{ineq_one}
\end{equation}

\noindent is fullfilled then $G$ is equitably   $t$-choosable for every integer $t$ satisfying $t\ge k$.
\end{lemma}

\begin{proof}
Let $k,t$ be fixed and $L$ be a $t$-uniform  list assignment for $G$. We show that there is a proper $L$-colouring $c$ of $G$  such that  the coloured graph $(G,c)$ has a  rainbow special $t$-partition. Since $L$ is chosen freely, it will follow that  $G$ is equitably  $t$-choosable, by Lemma \ref{lem:rainbow-colourable}. Let

\begin{itemize}
\item $|V(G)|=\eta k+r_1$, where $\eta,r_1$ are non-negative integers, $r_1\in [k]$,  and
\item $|V(G)|=\beta t+r_2$, where $\beta,r_2$ are non-negative integers, $r_2\in [t]$,  and  
\item $t=\gamma k+r$ where $\gamma,r$ are non-negative integers, $r\in [k]$.  
\end{itemize}
Thus $|V(G)|=\beta(\gamma k+r)+r_2=\beta\gamma k+\beta r+r_2$. We split  $V(G)$ into two subsets $V_1$ and $V_2$, where  $V_1=S_1\cup\cdots \cup S_{\eta+1-\beta\gamma }$ and $V_2=S_{\eta+1-(\beta\gamma-1)}\cup\cdots \cup S_{\eta+1}$. Observe that $|V_1|=\beta r+r_2$ and $|V_2|=\beta\gamma k$. First, we properly colour the vertices in  $V_1$, next we spread the colouring on $V_2$.  We colour vertices in each set $S_i$  of $V_1$ in such a way that we obtain a rainbow set. It is easy to see that we can colour vertices from $S_1$ such that we obtain a rainbow set, since   each vertex has assigned a list of length $t$ and $|S_1|=r_1\le t$. Next, we colour vertices $x_1^2,\ldots,x_k^2$ in $S_2$. We assign to $x_k^2$ a colour from its list that is not used  in $S_1$. Since $|N_G(x_k^2)\cap S_1|\le k-1$ and $|L(x_k)|=t\ge k$, this may be done. Next we assign to $x^2_{k-1},\ldots ,x^2_1$ (in the sequence) a colour from its list that is different from the ones assigned to the vertices with higher subscript and not used in $S_1$. All these steps may be completed since $|N(x_i^2)\cap S_1|\le i-1$ and $|L(x_i)|=t\ge k$. Similarly, we colour the vertices of each set $S_j\;(j\in[3 ,\eta+1-\beta\gamma ])$. Consider the   coloured subgraph $(G_1,c)$, where $G_1=G[V_1]$. Since each set $S_j\;(j\in[\eta+1-\beta\gamma] )$ is rainbow, we obtain a rainbow $k$-partition of $(G_1,c)$. 
If $r_2\le r_1$, we take $r_2$  vertices of $S_1$ and denote this set by $R$. Otherwise, we additionally choose  $r_2-r_1$ vertices from $S_2$ that have colours different than colours of vertices in $S_1$ and then these vertices together with $S_1$ form $R$. Observe that also $(G_1-R,c)$ has a rainbow $k$-partition. Furthermore, $|(S_1\cup\cdots \cup S_{\eta+1-\beta\gamma })\setminus R|=|V(G_1-R)|=\beta r$. By Lemma \ref{lem:rainbow_partition},  $G_1-R$ has a rainbow $r$-partition. Let $T_1,\ldots, T_{\beta }$ be a  rainbow $r$-partition of $(G_1-R,c)$. 

Now we colour the vertices in $V_2$. Recall that $|V_2|=\beta\gamma k$. Let us divide  $V_2$ into $\beta$ subsets, each containing $\gamma k$ sets $S_i$, in the following way:

$H_1=S_{\eta+1-(\beta\gamma -1)}\cup S_{\eta+1-(\beta\gamma -2)}\cup \cdots\cup S_{\eta+1-(\beta-1)\gamma }$

$H_2=S_{\eta+1-((\beta-1)\gamma -1)}\cup S_{\eta+1-((\beta-1)\gamma -2)}\cup \cdots\cup S_{\eta+1-(\beta-2)\gamma }$

$\vdots$

$H_i=S_{\eta+1-((\beta-i+1)\gamma -1)}\cup S_{\eta+1-((\beta-i+1)\gamma -2)\gamma+1}\cup \cdots\cup S_{\eta+1-(\beta-i)\gamma }$

$\vdots $

$H_{\beta }=S_{\eta+1-(\gamma-1)}\cup S_{\eta+1-(\gamma -2)}\cup \cdots \cup S_{\eta+1}$.

\noindent We will  properly colour  vertices in $H_1,\ldots ,H_\beta$ from their lists, step by step, in such a way that each set $T_i\cup H_i$ for $i\in[\beta]$ is rainbow. 

First, consider a colouring of vertices of $H_i$. 
To simplify the notation let $A=\alpha+1-((\beta-i+1)\gamma -1)$. Thus $H_i=S_{A}\cup S_{A+1}\cup\cdots\cup S_{A+\gamma-1}$. Recall that vertices $x_1^{A},\ldots ,x_k^{A}$ in $S_{A}$ fulfill the inequality (\ref{ineq_one}). We delete colours that are used on vertices in $T_i$ from lists of vertices in $S_{A}$. Now the lists of vertices in $S_{A}$ are shorter than $t$, however each vertex still has at least $\gamma k$ colours on the list. Assign to $x_k^{A}$ a colour from its list that is not used on vertices from $S_1\cup \cdots \cup S_{A-1}$. Since $|N_G(x_k^{A})\cap (S_1\cup \cdots \cup S_{A-1})|\le k-1$ and $|L(x_k^{A})|=\gamma k\ge k$, this may be done. Then assign to $x_{k-1}^{A},\ldots ,x_1^{A}$ (in a sequence) a colour from its list that is different from the ones assigned to the vertices with higher subscript and not used in $S_1\cup \cdots \cup S_{A-1}$. All these steps may be done since $|N_G(x_i^{A})\cap (S_1\cup \cdots \cup S_{A-1})|\le i-1$ and $|L(x_i^{A})|=\gamma k\ge k$. 
Now, we colour vertices in $S_{A+1}$, where $S_{A+1}=\{x_1^{A+1},\ldots ,x_k^{A+1}\}$. We delete  colours that are used on vertices in $T_i$ and $S_{A}$  from lists of vertices in $S_{A+1}$. Observe that after deleting colours from lists, each vertex in $S_{A+1}$ has at least $(\gamma-1)k$ colours on the list. Similarly as above, first we colour the vertex $x_k^{A+1}$ with a colour from its list that is not used  in $S_1\cup \cdots \cup S_{A}$ and then we colour, one by one,  vertices $x_{k-1}^{A+1},\ldots ,x_1^{A+1}$ with colours from their lists that are different from the ones assigned to the vertices with higher subscript and not used  in $S_1\cup \cdots \cup S_{A}$. We can do this since $|N_G(x_i^{A+1})\cap (S_1\cup \cdots \cup S_{A})|\le i-1$ and $|L(x_i^{A+1})|=(\gamma -1) k\ge k$. 
Observe  that in the same way we can colour vertices from sets $S_{A+2},\ldots,S_{A+\gamma-1}$.  Indeed, let $S_{A+j}=\{x_1^{A+j},\ldots ,x_k^{A+j}\}$. We delete from lists of vertices in $S_{A+j}$ colours that are used on vertices in $T_i\cup S_{A}\cup\cdots \cup S_{A+j-1}$ and then we assign the colour different from the ones assigned to the vertices with higher subscript and not used  in $S_1\cup \cdots \cup S_{A+j-1}$. 

Thus finally, we have obtained a proper colouring $c$ that admits a rainbow $t$-partition of $(G,c)$ which completes the proof.
\end{proof}

The next result generalizes Lemma \ref{lm}. We give only a sketch of its proof because it imitates  the proof of Lemma \ref{lem:vertex_partition}. 

\begin{lemma}
\label{lem:vertex_partition_arboricity}
Let  $k\in {\mathbb N}$.
If a graph $G$ has a special $k$-partition $S_1\cup \cdots \cup S_{\eta+1}$ such that $|S_1|\le k$ and $|S_j|=k$ for $j\in[2, \eta+1]$, moreover, if for every $j\in[2 , \eta+1]$ there is an ordering $x_1^j,\ldots ,x_k^j$ of vertices of  the set $S_j$ that for every $i\in[k]$ the inequality
\begin{equation}
|N_G(x_i^j) \cap (S_1\cup\cdots \cup S_{j-1})| \leq 2i-1,  \label{ineq_two} 
\end{equation}

\noindent is fulfilled then $G$ is equitably $t$-list  arborable  for any integer $t$ satisfying $t\ge k$. 
\end{lemma}

\begin{proof}
For fixed $k,t$ and a $t$-uniform  list assignment $L$ for $G$, we construct  an $L$-colouring $c$ of $G$  such that  the coloured graph $(G,c)$ has a  rainbow special $t$-partition and each colour class in $c$ induces an acyclic graph.  We do it in the same manner as in the proof of Lemma \ref{lem:vertex_partition}, but if we put a colour on the vertex $x_i^j$, $i\in[k],\;j\in [2,\eta+1]$ then we use Lemma \ref{lm} (instead of Lemma \ref{lem:Kostochka-Pelsmajer-West_2}) to guarantee that each colour class in $c$ induces an acyclic graph (instead of to guarantee that the constructed colouring is proper).
\end{proof}



Next, we give new tool that help us in proving further results concerning exact values as well as bounds on equitable list vertex arboricity of graphs.

A {\it spanning} graph $H$ of a graph $G$ is any subgraph of $G$ such that $V(H)=V(G)$. We say that a graph $H$ {\it covers all cycles} of $G$ if it is spanning and for any cycle $C$ contained in $G$ there are $x,y\in V(C)$ such that $xy\in E(H)$.

\begin{lemma}
\label{lem:covering}
Let $k\in \mathbb{N}$. If $H$ is a graph that covers all cycles of $G$ and $H$ is equitably $k$-choosable then $G$ is equitably $k$-list arborable.
\end{lemma}

\begin{proof}
Let $L$ be any $k$-list assignment for $G$. Let $c$ be an equitable proper $L$-colouring of $H$. We show that  each colour class induces an acyclic subgraph of $G$. Let $C$ be a cycle of $G$. By our assumption on $H$ there are $x,y\in V(C)$ such that $xy\in E(H)$. Thus $C$ contains two vertices which have different colours in $c$. Since $G$ has no monochromatic cycle in $c$, each colour class induces an acyclic graph.

\end{proof}

Lemma \ref{lem:covering} states that we can use known results related to equitable choosability for determining results  on equitable list vertex arboricity. Let us recall results proven in \cite{KiKo12}.

\begin{theorem}[\cite{KiKo12}]
\label{thm:Kierstead-Kostochka}
Let $r\in \mathbb{N}$ and  $G$ be a graph such that $\Delta(G) \le r$.

\begin{enumerate}[\rm (i)]
\item If $r\le 7$ and $k\ge r + 1$ then $G$ is equitably  $k$-choosable.
\item  If
$k\ge r+\left\{\begin{matrix}1+\frac{r-1}{7}& if & r\le 30 \cr \frac{r}{6} & if &r\ge 31\end{matrix}\right.$ then $G$ is equitably  $k$-choosable.
\item If $|V(G)|\ge r^3$ and   $k\ge r + 2$ then $G$ is equitably  $k$-choosable.
\item If $\omega(G)\le  r$ and $|V(G)|\ge 3(r + 1)r^8$ then $G$ is equitably  $(r + 1)$-choosable {\rm (}$\omega(G)$ is the clique number of $G${\rm )}.
\end{enumerate}
\end{theorem}

Theorem \ref{thm:Kierstead-Kostochka} and Lemma \ref{lem:covering} imply the general upper bound on equitable list vertex arboricity.

\begin{theorem}
\label{thm:general-upper-bound}
Let $r\in \mathbb{N}$ and $G$ be a graph with at least one edge and $\Delta(G) -1\le r$.

\begin{enumerate}[\rm (i)]
\item If $r\le 7$ and $k\ge r +1$ then $G$ is equitably $k$-list arborable.
\item  If
$k\ge r+\left\{\begin{matrix}1+\frac{r-1}{7}& if & r\le 30 \cr \frac{r}{6} & if &r\ge 31\end{matrix}\right.$ then $G$ is equitably $k$-list arborable.

\item If $|V(G)|\ge r^3$ and   $k\ge r + 2$ then $G$ is equitably $k$-list arborable.
\item If $\omega(G)\le  r$ and $|V(G)|\ge 3(r + 1)r^8$ then $G$ is equitably $(r + 1)$-list arborable.
\end{enumerate}
\end{theorem}

\begin{proof}
Let $F$ be a spanning forest of $G$ such that the numbers of connected components of $F$ and $G$ are the same. Thus $G-F$ covers all cycles of $G$. By Lemma \ref{lem:covering}, if $G-F$ is equitably  $k$-choosable then $G$ is equitably $k$-list arborable. Since $\Delta(G-F)\le \Delta(G)-1$, the theorem follows directly from Theorem \ref{thm:Kierstead-Kostochka}
\end{proof}

If we restrict our consideration to particular graph classes or to graphs with particular properties, we get even better bounds on equitable list arboricity that, in addition, confirm Zhang's conjecture. 


\begin{theorem}[\cite{KoPeWe03}]
\label{thm:Kostochka-Pelsmajer-West}
Let $k\in \mathbb{N}$ and let $F$ be a forest. If $k\ge \Delta(F)/2+1$ then $F$ is equitably  $k$-choosable.
\end{theorem}

We can apply Theorem  \ref{thm:Kostochka-Pelsmajer-West} to show an upper bound on equitable list vertex arboricity of graphs with (edge) arboricity equal to 2. The {\it $($edge$)$ arboricity} of a graph $G$ is the minimum number of forests into which its edges can be partitioned.

\begin{theorem}
\label{thm:arboricity_two}
Let $k\in \mathbb{N}$ and let $G$ be a graph with arboricity $2$. If $k\ge \left\lceil (\Delta(G)+1)/2\right\rceil$ then $G$ is equitably $k$-list arborable. 
\end{theorem}

\begin{proof}
Let $F_1=(V(G),E_1)$ and $F_2=(V(G),E_2)$ be two forests into which $E(G)$ was partitioned. Of course, $E(G)=E_1 \cup E_2$. It is clear that $F_1$ covers all cycles of $G$. If $\Delta(F_1)<\Delta(G)$ then by Theorem \ref{thm:Kostochka-Pelsmajer-West} and Lemma \ref{lem:covering} $G$ is equitably $k$-list arborable for $k\ge \Delta(F_1)/2+1$. It means that $G$ is equitably $k$-list arborable for $k\ge \left\lceil (\Delta(G)+1)/2\right\rceil$. Suppose that $\Delta(F_1)=\Delta(G)$. Let $D$ be the set of vertices of maximum degree in $F_1$. Observe that every vertex in $D$ is adjacent only with edges from $E_1$.  Let $E_1'\subseteq E_1$ be the minimal set of edges such that $D\subseteq \bigcup_{e\in E'_1}e$. Since $E_1'$ is minimal, the subgraph induced by $E'_1$ is a star-forest. Furthermore, in the subgraph induced by $E_2\cup E'_1$ every edge in $E'_1$ is a pendant edge. Thus the subgraph induced by $E_2\cup E'_1$ is acyclic and so $F_1-E'_1$ covers all cycles of $G$. Since $\Delta(F_1-E'_1)<\Delta(G)$,  by Theorem \ref{thm:Kostochka-Pelsmajer-West} and Lemma \ref{lem:covering}, $G$ is equitably $k$-list arborable for $k\ge \Delta(F_1-E'_1)/2+1$. It means that $G$ is equitably $k$-list arborable for  $k\ge \left\lceil (\Delta(G)+1)/2\right\rceil$. 
\end{proof}

A graph $G$ is $d$-{\it degenerate} if  every subgraph of $G$ has a vertex of degree at most $d$.
Since every $2$-degenerate graph has arboricity $2$, Theorem \ref{thm:arboricity_two} confirms the result for $2$-degenerate graphs obtained by Zhang \cite{zhang}.

\begin{corollary}[\cite{zhang}]
Let $k\in \mathbb{N}$ and let $G$ be a $2$-degenerate graph. If $k\ge \left\lceil (\Delta(G)+1)/2\right\rceil$ then $G$ is equitably $k$-list arborable. 
\end{corollary}

\section{Equitable choosability of grids}
\label{colouring-grids}

Since our new tool (Lemma \ref{lem:covering}) uses the notion of equitable choosability we dedicate this section to this notion for some graphs related to grids. Nethertheless, before we consider it, we give some sufficient conditions for graphs to be equitably $2$-choosable.

\begin{lemma}
\label{lem:two-equitable-list-colourable}
If $G$ has a matching of size $\left\lfloor |V(G)|/2\right\rfloor$ and $G$ is  $2$-choosable then $G$ is equitably  $2$-choosable.
\end{lemma}

\begin{proof}
Observe that the assumption that $G$ has a matching of size $\left\lfloor |V(G)|/2\right\rfloor$ implies  that $\alpha(G)\le \left\lceil |V(G)|/2\right\rceil$ ($\alpha(G)$ denotes  the cardinality of the largest independent vertex set of $G$). Thus  each colour class has at most $\left\lceil |V(G)|/2\right\rceil$ vertices in any proper colouring of $G$. Let $L$ be a $2$-uniform  list assignment for $G$. Since $G$ is  2-choosable, there is a proper $L$-colouring $c$ of $G$. Furthermore, every colour class in $c$ has at most $\left\lceil |V(G)|/2\right\rceil$ vertices, and so $c$ is equitable proper $L$-colouring of $G$. 
\end{proof}

The graphs that are 2-choosable were characterized by Erd\"{o}s, Rubin and Taylor in \cite{ErRuTa80}.
The \emph{core} of  $G$ is  a graph obtained from $G$ by recursive removing all vertices of degree one. Thus the core of $G$ has no vertices of degree one. A graph is called a $\Theta_{2,2,p}$-\emph{graph} if it consists of two vertices $x$ and $y$ and three internally disjoint paths of lengths 2, 2 and $p$, joining $x$ and $y$. 

\begin{theorem}[\cite{ErRuTa80}]
\label{thm:two-list-colourable}  
A connected graph $G$ is $2$-choosable if and only if the core of $G$ is either $K_1$, or an even cycle, or a $\Theta_{2,2,2r}$-graph, where $r\in \mathbb{N}$.
\end{theorem}

\begin{lemma}
\label{lem:equitable-colouring-cycles}
Let $k \in \mathbb{N}$ with $k\geq 2$. If $G$ is a bipartite graph with $\Delta(G)\le 2$ then $G$ is equitably  $k$-choosable.
\end{lemma}

\begin{proof}
Observe first that each component of $G$ is either an even cycle or a path. If  $G$ has more than one component that is a path, let $G'$ be a graph obtained from $G$ by adding  edges so that $G'$ has  one component that is a path and all other components are even cycles. In the case when $G$ has at most one component that is a path, we assume $G'=G$. We will show that $G'$ is equitably  $k$-choosable for any $k\ge 2$. By Theorem \ref{thm:two-list-colourable}, being applied to each connected component of $G'$, $G'$ is 2-choosable (it is clear that if each component is 2-choosable then the whole graph is also 2-choosable). Since  $G'$ has a matching of size $\left\lfloor |V(G')|/2\right\rfloor$ then $G'$ is equitably  $2$-choosable by Lemma \ref{lem:two-equitable-list-colourable}. Furthermore,  Theorem \ref{thm:Kierstead-Kostochka}(i) follows that $G'$ is equitably  $k$-choosable for every $k\ge 3$ (since $\Delta(G')\le 2$). Hence the arguments that $G'$ is equitably  $k$-choosable for any $k\ge 2$ and that $G$ is a spanning subgraph of $G'$ imply that $G$ is equitably  $k$-choosable for any $k\ge 2$. 
\end{proof}


Now, we define ${\cal G}_{1}$ to be a family of all grids $P_{n_1}\square P_2$ and all graphs resulting from  grids $P_{n_1}\square P_2$ by removing one vertex of minimum degree, taken over all $n_1\in \mathbb{N}$. The following results will be used in the next section to determine equitable list vertex arboricity of grids.


\begin{lemma}
\label{lem:equitable-colouring-p_n_one-p_two}
Let $k \in \mathbb{N}$ with $k\geq 3$. If every component of a graph $G$ is in ${\cal G}_{1}$ then $G$  is equitably  $k$-choosable.
\end{lemma}

\begin{proof}
We show that there is a special $3$-partition of $G$ that fulfills the assumptions of Lemma \ref{lem:vertex_partition}, i.e. there are disjoint sets $S_1, \ldots ,S_{\eta+1}$ such that the following conditions hold

\begin{itemize}
\item $V(G)=S_1\cup \cdots \cup S_{\eta+1}$; 
\item $|S_1|\le 3$ and $|S_j|=3$ for $j \in[2, \eta+1]$; 
\item there is an ordering of vertices of each set $S_j$, say $x^j_1,x^j_2 ,x^j_3$, fulfilling the inequality $|N_G(x_i^j) \cap (S_1\cup\ldots \cup S_{j-1})| \leq i-1$ for $i\in [3]$,
\end{itemize} 
and hence, by Lemma \ref{lem:vertex_partition}, $G$  is equitably  $k$-choosable for any $k\ge 3$.
We prove the existence of the partition by induction on the number of vertices of $G$. It is easy to see that it is true for a graph with at most 3 vertices. Thus suppose that if every component of a graph is in ${\cal G}_{1}$ and the graph has  less than $n$ vertices, $n\geq 4$, then  it has  a special $3$-partition that fulfills the assumptions of Lemma \ref{lem:vertex_partition}. Let $G$ be an $n$-vertex graph having every component in ${\cal G}_{1}$. We show that there is a set  $S$ in $G$, say $\{x_1,x_2,x_3\}$, such that $|N_G(x_i)\setminus S| \leq i-1$ for $i\in [3]$ and every component of $G-S$ is in ${\cal G}_{1}$. Thus, by induction, the lemma  follows. 

Let $x_1$ be a vertex of the minimum degree in $G$, thus $\deg_G(x_1)\le 2$. Suppose first that $\deg_G(x_1)=2$. In this case each component has at least four vertices. Let $x_2,x_3$ be the neighbors of $x_1$ such that $\deg_G(x_2)=2$ and $\deg_G(x_3)\le 3$. Let  $S=\{x_1,x_2,x_3\}$, then $|N_G(x_1)\setminus S| = 0,|N_G(x_2)\setminus S| \leq 1$ and $|N_G(x_3)\setminus S| \leq 2$. Observe that every component of $G-S$ is in ${\cal G}_{1}$, so by our induction hypothesis $G-S$ has a special $3$-partition that fulfills the assumptions of Lemma \ref{lem:vertex_partition}, and so we are done.

Suppose now that  $\deg_G(x_1)=1$. Let $x_2$ be the neighbor of $x_1$. If $\deg_G(x_2)=3$ then let $x_3$ be the neighbor of $x_2$ of degree 2. Let $S=\{x_1,x_2,x_3\}$. Hence every component of $G - S$ is in ${\cal G}_{1}$ and we see that the vertices of $S$ satisfy $|N_G(x_i)\setminus S| \leq i-1$ for $i\in [3]$. If $\deg_G(x_2)=2$ then let $x_3$ be the neighbor of $x_2$, other than $x_1$. Observe that in this case the vertices $x_1,x_2,x_3$ form a component of $G$. Again  $S=\{x_1,x_2,x_3\}$ satisfies $|N_G(x_i)\setminus S|=0 \leq i-1$ for $i\in [3]$, and so, by induction hypothesis, $G$ has a special $3$-partition that fulfills the assumptions of Lemma \ref{lem:vertex_partition}. If  $\deg_G(x_2)=1$ then as  $x_3$ in $S$ we  put a vertex of the minimum degree in $G- \{x_1,x_2\}$. 

Finally suppose that  $\deg_G(x_1)=0$. In this case let $x_2,x_3$ be two adjacent vertices of degree at most two. If there are no such vertices then $G$ is an edgeless graph and we can choose  $x_2,x_3$ arbitrarily. Similarly as above we  can see that every component of $G - S$ is in ${\cal G}_{1}$ and $S$ satisfies $|N_G(x_i)\setminus S| \leq i-1$ for $i\in [3]$. It implies that $G$ has a special $3$-partition that fulfills the assumptions of Lemma \ref{lem:vertex_partition}, and so $G$ is equitably   $k$-choosable for $k\ge 3$. 
\end{proof}

It should be mentioned here that for each component of graph $G$ in ${\cal G}_1$, we have  $\Delta (G)\le 3 $. Thus, by Theorem \ref{thm:Kierstead-Kostochka}, such a graph is equitably $k$-choosable for $k\ge 4$. 
Hence Lemma~\ref{lem:equitable-colouring-p_n_one-p_two} extends this result to $k\geq 3$.

Let $n_1,n_2\in \mathbb{N}$, $n_2 \geq 2$, and $\ell\in [0,n_1-1]$. The symbol $(P_{n_1}\square P_{n_2},\ell)$ denotes a graph obtained from $P_{n_1}\square P_{n_2}$ by the delation of a set $V^\prime$ (cf. Fig.~\ref{fig:familyG2}), where 
$$V^\prime=\{(n_1-p,n_2):\; p\in [0,\ell-1]\}\cup \{(n_1-p,n_2-1):\; p\in [0,\ell-1]\}.
$$


\noindent Observe  that $(P_{n_1}\square P_{n_2},0)$ is a grid $P_{n_1}\square P_{n_2}$.

Let ${\cal G}_{2}=\{(P_{n_1}\square P_{n_2},\ell):n_1\ge 1,n_2\ge 2,\ell\in [0, n_1-1]\}$.

\begin{figure}
    \centering
    \includegraphics{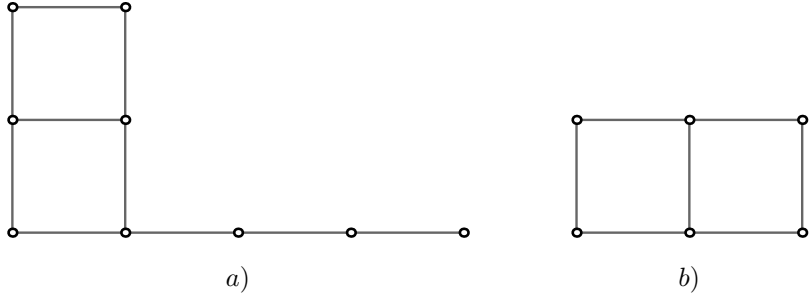}
    \caption{A graph a)$(P_5 \square P_3,3)$ and b) $(P_5 \square P_2,2)$ being isomorphic to $P_3 \square P_2$.}
    \label{fig:familyG2}
\end{figure}

\begin{lemma}
\label{lem:equitable-colouring-p_n_one-p_n_two}
Let $k \in \mathbb{N}$ with $k\geq 4$.
If  each component of a graph $G$ is in ${\cal G}_{2}$ then $G$  is equitably  $k$-choosable.
\end{lemma}

\begin{proof}
We show that there is a special $4$-partition of $G$ that fulfills the assumptions of Lemma \ref{lem:vertex_partition}.
We prove it by induction on the number of vertices. Observe that every graph in ${\cal G}_{2}$ has at least two vertices and it is easy to see that if $G$ has at most 4 vertices then $G$ has a special $4$-partition  that fulfills the assumptions of Lemma \ref{lem:vertex_partition}. Suppose that the assertion is true for graphs with less that $n$ vertices, $n\geq 5$. Let $G$ be a graph with $n$ vertices that satisfies assumptions of the lemma. We show that  there is a set $S$, say $\{x_1,x_2,x_3,x_4\}$, such that $|N_G(x_i)\setminus S| \leq i-1$ for $i\in [4]$ and each component of $G-S$ is in ${\cal G}_{2}$. 

We choose the set $S$ as follows. 
First suppose that there is a component  $(P_{n_1}\square P_{n_2},\ell)$  of $G$ such that $n_1-\ell\ge 2$ and $n_2\ge 2$. Let us consider the set $S=\{x_1,x_2,x_3,x_4\}$ with $x_1=(n_1-\ell,n_2)$, $x_2=(n_1-\ell-1,n_2)$, $x_3=(n_1-\ell,n_2-1)$, and $x_4=(n_1-\ell-1,n_2-1)$. 

\noindent Thus 
$|N_G((n_1-\ell,n_2))\setminus S|= 0,|N_G((n_1-\ell-1,n_2))\setminus S|\le 1,|N_G((n_1-\ell,n_2-1))\setminus S|\le 1$ and $|N_G((n_1-\ell-1,n_2-1))\setminus S|\le 2$.
Furthermore, every component of $G-S$ is in ${\cal G}_{2}$ and hence, by the induction hypothesis, $G$ has $4$-partition of $G$ that fulfills the assumptions of Lemma \ref{lem:vertex_partition}.
If there is a component   $(P_{n_1}\square P_{n_2},\ell)$ of $G$ such that $n_1-\ell=1$ and $n_2\ge 4$ then we put $x_1=(1,n_2)$, $x_2=(1,n_2-1)$, $x_3=(1,n_2-2)$, and $x_4=(1,n_2-3)$. 
Every component of $G-S$ is in ${\cal G}_{2}$ and 
$|N_G((1,n_2))\setminus S|= 0,|N_G((1,n_2-1))\setminus S|=0,|N_G((1,n_2-2))\setminus S|\le 1$ and $|N_G((1,n_2-3))\setminus S|\le 2$, so by the induction hypothesis, the assumptions of Lemma \ref{lem:vertex_partition}  are satisfied.
Otherwise, every component of $G$ is a path. If there is a component with at least four vertices then  four consecutive vertices of the path form 
the set $S$ that satisfies $|N_G(x_i)\setminus S| \leq i-1$ for $i\in [4]$. If each component of $G$ has less than four vertices then, to obtain $S$, we  take all vertices of one component and we next complete the set $S$ by vertices of some other component or even components, if the number of vertices chosen to set $S$ is still to small. It is easy to see that also in such a case the assumptions of Lemma \ref{lem:vertex_partition} are fulfilled, which finishes the proof.
\end{proof}

\begin{lemma}
\label{lem:equitable-colouring-p_n_one-p_n_two-with_three}
Let $n_1,n_2, t\in \mathbb{N}$. If $G$ is a graph with $t$ components such that each one is isomorphic to $P_{n_1}\square P_{n_2}$ then $G$  is equitably  $3$-choosable.
\end{lemma}

\begin{figure}[htb]
\begin{center}
\includegraphics[scale=1]{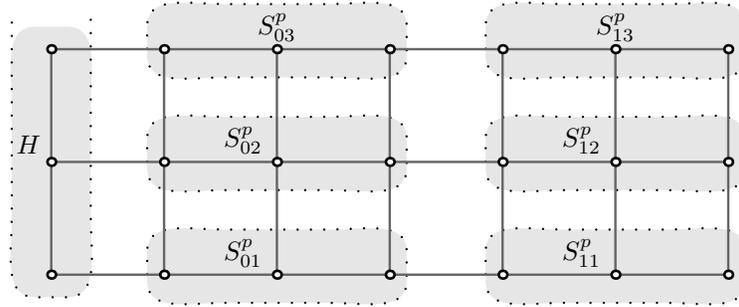}
\caption{Illustration for the proof of Lemma \ref{lem:equitable-colouring-p_n_one-p_n_two-with_three}; $G^p=P_7 \square P_3$.}
\label{two_dimensional_three-colouring}
\end{center}
\end{figure}

\begin{proof}
If $n_1\leq 2$ or $n_2\leq 2$ then the proof follows from Lemma \ref{lem:equitable-colouring-p_n_one-p_two}. Thus we may assume that $n_1\ge 3$ and $n_2\ge 3$. Let $G^p=P^p_{n_1}\square P^p_{n_2}$ for $p\in[t]$ be components of $G$ and $\{(i,j)^p:i\in[n_1],j\in[n_2]\}$ be the vertex set of the component $G^p$. Let $n_1=3q+r$ where $ r\in [0,2]$ and  let $L$ be a $3$-uniform  list assignment for the graph $G$. We show that there is a proper $L$-colouring $c$  such that $(G,c)$ has a rainbow  3-partition. Let $H$ be a subgraph of $G$ induced by the set $\{(1,j)^p:j\in[n_2],p\in[t]\}$ if $r=1$, and induced by the set $\{(1,j)^p,(2,j)^p:j\in[n_2],p\in[t]\}$ if $r=2$. 
Moreover, let $S^p_{ij}=\{(3i+1+r,j)^p,(3i+2+r,j)^p,(3i+3+r,j)^p\}$ where $i\in [0,q-1], j\in [n_2],p\in[t]$ (cf. Fig.~\ref{two_dimensional_three-colouring}).

First, we colour the vertices of $H$. Let $c^{\prime}$ be an equitable proper $L$-colouring  of $H$ guaranteed by Lemma \ref{lem:equitable-colouring-p_n_one-p_two}. Thus, by Lemma \ref{lem:rainbow-colourable}, there is a rainbow  3-partition of $(H,c^{\prime})$. After this step all vertices of the first and the second column are coloured if $r=2$, all vertices of the first column are coloured if $r=1$, and graph is uncoloured if $r=0$.  Next, in each component,  we colour uncoloured vertices of the first row, i.e., $(r+1,1)^p,(r+2,1)^p,\ldots,(n_1,1)^p$ for $p\in[t]$. We properly colour these vertices in such a way that the sets $S_{i1}^p$, $i\in [0,q-1]$ are rainbow. Now  we divide the uncoloured vertices of each component into 3-element subsets $S^p_{ij}$ where $i\in [0,q-1], j\in [2,n_2]$, and $p\in[t]$.
In each component we define linear ordering $\prec ^p$ on these sets in the following way: $S^p_{ij}\prec S^p_{rs}$ if ($j<s$) or ($j=s$ and $i<r$). According to this ordering, we properly colour vertices of each set $S_{ij}^p$ with the following rules.
\begin{itemize}
    \item If it is only possible, we colour vertices in $S^p_{ij}$ in such a way that vertices of this set obtain different colours.
    \item If we cannot colour vertices in $S^p_{ij}$ in such a way that $S_{ij}^p$ is rainbow then we color vertices in $S^p_{ij}$ in such a way that
    two vertices have the same colour, let us say $c_1$, and there is no vertex coloured with $c_1$ in $S^p_{ij-1}$. Moreover, if also
    the set $S^p_{ij-1}$ is not rainbow, i.e. two vertices in $S^p_{ij-1}$ are coloured with the same colour, let us say $c_2$, then there is no vertex coloured with $c_2$ in $S^p_{ij}$.
\end{itemize}

We will show that such a colouring exists. Let $c^{\prime\prime}$ be a proper $L$-colouring of $G-H$ such that these rules are maintained. 
Suppose that we are at the step when we have just coloured vertices in $S^p_{ij}$, so vertices in every set that precedes $S^p_{ij}$, with respect to $\prec^p$, and the vertices in $S^p_{ij}$ are coloured, the vertices in $S^p_{i+1j}$ (or $S^p_{0j+1}$ when $i=q-1$) are uncoloured. To simplify notation let $S^p_{ij}=\{(x,j),(x+1,j),(x+2,j)\}$ and $S^p_{ij-1}=\{(x,j-1),(x+1,j-1),(x+2,j-1)\}$. Let $c^{\prime\prime}((x,j))=c^{\prime\prime}((x+2,j))=c_1$ and $c^{\prime\prime}((x+1,j))=b_1$. First we show that there is no vertex coloured with $c_1$ in $S^p_{ij-1}$. Since vertices $(x,j)$ and $(x,j-1)$ are adjacent, it follows that $c_1\neq c^{\prime\prime}((x,j-1))$. Similarly, $c_1\neq c^{\prime\prime}((x+2,j-1))$.  Now we need to show that $c_1\neq c^{\prime\prime}((x+1,j-1))$. Since we use $c_1$ to colour $(x+2,j)$ then we necessarily have $L((x+2,j))=\{c_1,b_1,c^{\prime\prime}((x+2,j-1))\}$. If $c_1= c^{\prime\prime}((x+1,j-1))$ then we could colour $(x+1,j)$ with colour different from $c_1$ and $b_1$ and next colour $(x+2,j)$ with $b_1$ and so we would colour the vertices in $S^p_{ij}$ with different colours, a contradiction. To finish the reasoning we show that if $c^{\prime\prime}((x,j-1))=c^{\prime\prime}((x+2,j-1))=c_2$ then  there is no vertex coloured with $c_2$, in $S^p_{ij}$. It is easy to see that $c_2\neq c^{\prime\prime}((x,j))$  and $c_2\neq c^{\prime\prime}((x+2,j))$. As we observed above $L((x+2,j))=\{c_1,b_1,c^{\prime\prime}((x+2,j-1))\}$.
Since each vertex has  the list consisting of three different colours, we have $b_1\neq c^{\prime\prime}((x+2,j-1))$ and so $c^{\prime\prime}((x+1,j))\neq c_2$.

\begin{figure}[htb]
\begin{center}
\includegraphics[scale=1]{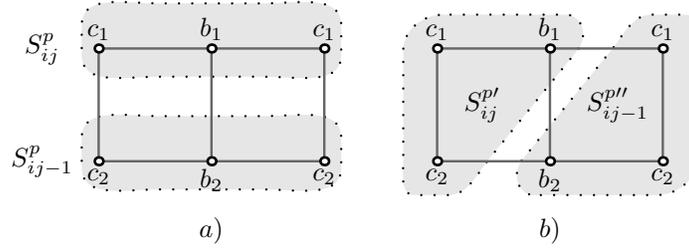}
\caption{a) A part of $G^p$ with depicted non-rainbow sets $S^p_{ij}$ and $S^p_{ij-1}$, $b_2 \neq c_1$, $b_1 \neq c_2$. b) Repartition of $S^p_{ij} \cup S^p_{ij-1}$ into two rainbow sets $S^{p\prime}_{ij}$ and $S^{p\prime\prime}_{ij-1}$.}
\label{two_dimensional_three-colouringClaim}
\end{center}
\end{figure}

Above described rules imply that either $S^p_{ij}$ is rainbow or $S^p_{ij}\cup S^p_{ij-1}$ can be divided into two 3-element rainbow sets in $(G-H,c^{\prime\prime})$: $S^{p\prime}_{ij}\cup S^{p\prime\prime}_{ij-1}$ (cf. Fig.  \ref{two_dimensional_three-colouringClaim}). We use this property to show that there is a rainbow special 3-partition of $(G-H,c^{\prime\prime})$. We divide $V(G-H)$ in the following way:
\begin{itemize}
\item We divide the vertices of each component step by step.
\item In each component $G^p$, we start with the last set, with respect to $\prec^p$, and go down due to this ordering.
\item  If $S^p_{ij}$ is rainbow then it forms a set of the rainbow special 3-partition of $(G-H,c^{\prime\prime})$. Otherwise, we partite $S^p_{ij}\cup S^p_{ij-1}$ into  two $3$-element rainbow sets $S^{p\prime}_{ij}\cup S^{p\prime\prime}_{ij-1}$ (cf. Fig. \ref{two_dimensional_three-colouringClaim}). We modify $\prec ^p$ by removing  sets that are already included in the rainbow $3$-partition.
\end{itemize}

Recall that  the sets $S^p_{i1}$ for $i\in [0,q-1]$ (sets of the first row) are rainbow, so the above partition results in a rainbow special 3-partition of $(G-H,c^{\prime\prime})$. Thus together with the rainbow  3-partition of $(H,c^{\prime})$ we obtain the rainbow  3-partition of $(G,c^{\prime}\cup c^{\prime\prime})$. Hence for every  $3$-uniform list assignment $L$ there is a proper $L$-colouring $c$ such that $(G,c)$ has a rainbow  3-partition and next, by Lemma \ref{lem:rainbow-colourable},  $G$ is equitably  $3$-choosable.
\end{proof}

Lemma \ref{lem:equitable-colouring-p_n_one-p_n_two} and Lemma \ref{lem:equitable-colouring-p_n_one-p_n_two-with_three} immediately imply the following result.
\begin{lemma}
\label{lem:equitable-colouring-p_n_one-p_n_two-final}
Let  $n_1,n_2,k\in \mathbb{N}$ with $k\geq 3$. If  each component of a graph $G$ is isomorphic to $P_{n_1}\square P_{n_2}$ then $G$  is equitably  $k$-choosable.
\end{lemma}

If each component of graph $G$ is in $P_{n_1}\square P_{n_2}$ then $\Delta (G)\le 4 $. Thus, by Theorem \ref{thm:Kierstead-Kostochka}, such a graph is equitably $k$-choosable for $k\ge 5$. Hence Lemma \ref{lem:equitable-colouring-p_n_one-p_n_two-final} extends this result to $k\ge 3$.

\begin{remark}
Observe that {\rm Lemma \ref{lem:equitable-colouring-p_n_one-p_n_two-with_three}} and {\rm Lemma \ref{lem:equitable-colouring-p_n_one-p_n_two-final}} are  still true if each component of $G$ is an arbitrary  $2$-dimensional  grid {\rm (}components are not necessarily of the same sizes{\rm )}. Furthermore, the bound in {\rm Lemma \ref{lem:equitable-colouring-p_n_one-p_n_two-final}} is tight, since $P_2\square P_3$ is not  $2$-choosable.
\end{remark}


\begin{lemma}
\label{lem:equitable-colouring-p_n_one-p_n_two-p_two}
Let $n_1,n_2\in \mathbb{N}$ and $t,s \in \mathbb{N}_0$. If $G$ is a graph with $t$ components such that each one is isomorphic to $P_{n_1}\square P_{n_2}\square P_2$ and with $s$ components being isomorphic to $P_{n_1}\square P_{n_2}$ then $G$  is equitably  $4$-choosable.
\end{lemma}

\begin{figure}[htb]
\begin{center}
\includegraphics[scale=1]{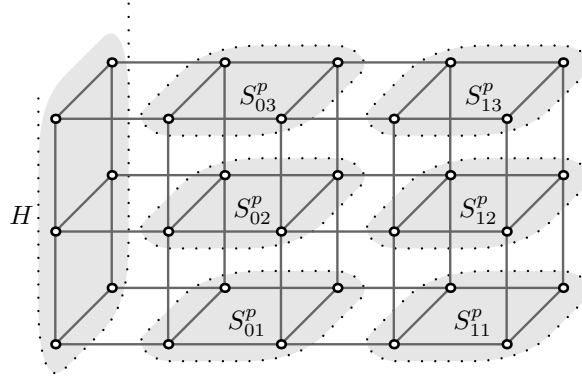}
\caption{Illustration for the proof of Lemma \ref{lem:equitable-colouring-p_n_one-p_n_two-p_two}; $G^p=P_5\square P_3\square P_2$.}
\label{three_dimensional_four-colouring}
\end{center}
\end{figure}

\begin{proof}
If  $n_1=1$ or $n_2=1$ then the proof follows from  Lemma \ref{lem:equitable-colouring-p_n_one-p_n_two-final}. Thus, without loss of generality, we may assume that $n_1, n_2\ge 2$. Let $G^p=P_{n_1}\square P_{n_2}\square P_2,\;F^u=P_{n_1}\square P_{n_2}$ for $p\in[t], u\in[s]$ be components of $G$ and 

\noindent $V(G^p)=\{(i,j,\ell)^p:i\in[n_1],j\in[n_2],\ell\in[2]\}$, 

\noindent $V(F^u)=\{(i,j)^u:i\in[n_1],j\in[n_2]\}$. 

\noindent Let $n_1=2q+r$ where $r\in [0,1]$.
Let $L$ be a $4$-uniform  list assignment for a graph $G$. We show that there is a proper $L$-colouring $c$  such that $(G,c)$ has a rainbow 4-partition. If $r=1$ then  let $H$ be a subgraph induced in $G$ by the set $\{(1,j,\ell)^p:j\in[n_2],p\in[t],\ell\in[2]\}\cup\{(i,j)^u:i\in[n_1],j\in[n_2],u\in[s]\}$. If $r=0$ then  let $H$ be a subgraph induced in $G$ by the set $\{(i,j)^u:i\in[n_1],j\in[n_2],u\in[s]\}$. By Lemma \ref{lem:equitable-colouring-p_n_one-p_n_two-final} there is an equitable proper $L$-colouring $c^{\prime}$ of $H$, and so by Lemma \ref{lem:rainbow-colourable} there is a rainbow 4-partition of $(H,c^{\prime})$. Now   we start with colouring vertices of $G-H$ (vertices of $G$, if $r=0$ and $G$ has no component isomorphic to $P_{n_1}\square P_{n_2}$). 

We divide the set of uncoloured vertices of each component  into 4-element subsets. 

\noindent $S^p_{ij}=\{(2i+1+r,j,1)^p,(2i+1+r,j,2)^p,(2i+2+r,j,1)^p,(2i+2+r,j,2)^p\}$, 
where $i \in [0,q-1], j\in[n_2],p\in[t]$ (cf. Fig~\ref{three_dimensional_four-colouring}).
In each component we define a linear ordering $\prec ^p$ on the family of these sets in the following way: $S^p_{ij}\prec S^p_{rs}$ if ($j<s$) or ($j=s$ and $i<r$). According to this ordering we properly colour vertices of each set with the following rules. 
\begin{itemize}
\item If it is only possible, we colour vertices in $S^p_{ij}$ in such a way that the vertices from this set get different colours. 
\item If we cannot colour vertices in $S^p_{ij}$ in such a way that $S_{ij}^p$ is rainbow then we colour vertices in this set in such a way that two vertices have the same colour, let us say colour $c$, other vertices are coloured differently and  there is no vertex coloured with $c$ in $S^p_{ij-1}$.
\end{itemize}

We show that there exists a proper $L$-colouring of $G-H$ such that these rules are maintained. It is easy to see that we can colour vertices in sets $\{S^p_{i1}:i\in[0,q-1]\}$ such that these sets are rainbow.
Suppose that we are at the step when we  colour vertices in $S^p_{ij}$, $j\geq 2$, so vertices of every set that precedes $S^p_{ij}$  are coloured, the vertices in $S^p_{ij}$ are uncoloured. Let $c^{\prime\prime}$ be a proper  $L$-colouring of the coloured part of $G-H$ constructed up to now.  To simplify the notation let $S^p_{ij}=\{(x,j,1),(x,j,2),(x+1,j,1),(x+1,j,2)\}$. Thus each  vertex in $\{(x,j,1),(x,j,2)\}$ has at most two coloured neighbours that are not in $S^p_{ij}$ and each vertex in $\{(x+1,j,1),(x+1,j,2)\}$ has one coloured neighbour that is not in $S^p_{ij}$. Suppose that we cannot colour vertices in $S^p_{ij}$ such that $S^p_{ij}$ is rainbow. Since every vertex has four colours on its list, we can always colour three vertices in $S^p_{ij}$  with different colours, only the last  vertex being coloured in $S^p_{ij}$ obtains the colour just used  on  $S^p_{ij}$. Let $c^{\prime\prime}((x,j,1))=c_1,c^{\prime\prime}((x,j,2))=c_2,c^{\prime\prime}((x+1,j,1))=c_3,c^{\prime\prime}((x+1,j,2))=c_1$. If there is no vertex coloured with $c_1$ in $S^p_{ij-1}$ then we are done. Suppose that there is a vertex coloured with $c_1$ in $S^p_{ij-1}$. Since we are forced to use the colour $c_1$ on $(x+1,j,2)$, we necessarily have $L((x+1,j,2))=\{c_1,c_2,c_3,c^{\prime\prime}(x+1,j-1,2)\}$. If in $L((x+1,j,1))$ there is a colour $b$ such that  $b\notin\{c_1,c_2,c_3,c^{\prime\prime}((x+1,j-1,1))$ then we can colour $(x+1,j,1)$ with $b$ and next we colour $(x+1,j,2)$ with $c_3$, to obtain a rainbow set  $S^p_{ij}$, a contradiction. Thus $L((x+1,j,1))=\{c_1,c_2,c_3,c^{\prime\prime}(x+1,j-1,1)\}$. Since each vertex has  four different colours on the list, we have $c_1\neq c^{\prime\prime}(x+1,j-1,2)$ and $c_1\neq c^{\prime\prime}(x+1,j-1,1)$. Furthermore, $(x,j-1,1)$ has a neighbour coloured with $c_1$, thus $c^{\prime\prime}((x,j-1,1))\neq c_1$. However,  by our assumption in $S^p_{ij-1}$ there is a vertex coloured with $c_1$, so $c^{\prime\prime}((x,j-1,2))=c_1$. Observe that also $c_2\neq c^{\prime\prime}((x+1,j-1,2))$ and $c_2\neq c^{\prime\prime}((x+1,j-1,1))$. Thus if $c_2\neq c^{\prime\prime}((x,j-1,1))$ then we can colour $(x+1,j,1)$ with $c_2$ and $(x+1,j,2)$ with $c_3$ to obtain desired colouring. Assume that $c_2= c^{\prime\prime}((x,j-1,1))$. Observe that there is no vertex coloured with $c_3$ in $S^p_{ij-1}$. If $c_3\in L((x,j,1))$ then we colour $(x,j,1)$ with $c_3$ and next $(x+1,j,1)$ with $c_2$ to obtain a desired colouring. 
Otherwise, $(x,j,1)$ has a colour $b$ different from $c_1,c_2,c_3$ and $c^{\prime\prime}(x-1,j,1)$ on its list. If we colour $(x,j,1)$ with $b$, the $S^p_{ij}$ is rainbow, a contradiction.

\begin{claim}
\label{claim:rainbow_sets_two}
If the set $S^p_{ij}$ is not rainbow and $S^p_{ij-1}$ is not rainbow, i.e., in $S^p_{ij-1}$ there are two vertices coloured with $b_1$, then in $S^p_{ij}$ there is no vertex coloured with $b_1$.
\end{claim}

\begin{proof}
Without loss of generality we may assume $c^{\prime\prime}((x,j,1))=c_1,c^{\prime\prime}((x,j,2))=c_2,c^{\prime\prime}((x+1,j,1))=c_3,c^{\prime\prime}((x+1,j,2))=c_1$. 
Similarly as above we  observe that $L((x+1,j,1))=\{c_1,c_2,c_3,c^{\prime\prime}((x+1,j-1,1))\}$ and $L((x+1,j,2))=\{c_1,c_2,c_3,c^{\prime\prime}((x+1,j-1,2))\}$. Since the colours on lists are different,  $c^{\prime\prime}((x+1,j-1,1))\notin \{c_1,c_2,c_3\}$ and $c^{\prime\prime}(x+1,j-1,2)\notin \{c_1,c_2,c_3\}$ and hence neither $c^{\prime\prime}((x+1,j-1,1))$ nor  $c^{\prime\prime}((x+1,j-1,2))$ is used on $S^p_{ij}$. The argument that $c^{\prime\prime}((x,j-1,1))\neq c^{\prime\prime}((x,j-1,2))$ completes the proof.
\end{proof}

Previous arguments imply that either $S^p_{ij}$ is rainbow or $S^p_{ij}\cup S^p_{ij-1}$ can be divided into two 4-elements rainbow sets in $(G-H,c^{\prime\prime})$, as it has been shown that each colour is used in $S^p_{ij}\cup S^p_{ij-1}$ at most twice.

We use the similar method as in the proof of Lemma \ref{lem:equitable-colouring-p_n_one-p_n_two-with_three} to show that there is a rainbow 4-partition of $(G-H,c^{\prime\prime})$. We divide $V(G-H)$ in the following way (cf. Fig. \ref{three_dimensional_four-colouring}):

\begin{itemize}
\item We divide the set of vertices of each component step by step.
\item In each component $G^p$, we start with the last set due to $\prec ^p$ and go down according this ordering.
\item  If $S^p_{ij}$ is rainbow then it forms a set of the rainbow special 4-partition of $(G-H,c^{\prime\prime})$. Otherwise, we partite $S^p_{ij}\cup S^p_{ij-1}$ into two rainbow 4-element sets  that form two sets of the rainbow 4-partition of $(G-H,c^{\prime\prime})$. We modify $\prec ^p$ by removing  sets that have been already included into the rainbow 4-partition.
\end{itemize}

Recall that for $i\in[0,q-1]$ the sets $S^p_{i1}$  are rainbow, so the above partition results in a rainbow special 4-partition of $(G-H,c^{\prime\prime})$. Thus, together with the rainbow 4-partition of $(H,c^{\prime})$, we obtain the rainbow 4-partition of $(G,c^{\prime}\cup c^{\prime\prime})$. Hence for every  $4$-uniform list assignment $L$ there is a proper $L$-colouring $c$ such that $(G,c)$ has a rainbow 4-partition, and so   $G$ is equitably  4-choosable, by Lemma  \ref{lem:rainbow-colourable}.

\end{proof}

\begin{remark}
Lemma \emph{\ref{lem:equitable-colouring-p_n_one-p_n_two-p_two}} is still true when components of $G$ are of different size.
\end{remark}
Observe that the $4$-partition given in the proof of Lemma \ref{lem:equitable-colouring-p_n_one-p_n_two-p_two} does not meet the assumptions of Lemma \ref{lem:vertex_partition}, thus from that proof we cannot conclude that such a graph is equitably  $k$-choosable for $k>4$. However, if each component of $G$ is isomorphic to $P_{n_1}\square P_{n_2}\square P_2$ or $P_{n_1}\square P_{n_2}$ then $\Delta(G) \leq 5$ and by Theorem \ref{thm:Kierstead-Kostochka} we have that $G$ is equitably  $k$-choosable for $k\ge 6$.


\section{Equitable list vertex arboricity of grids}
\label{sec:arboricity_grids}

In this section we apply tools described in the previous sections what causes in giving new results concerning equitable list arboricity of $d$-dimensional grids $P_{n_1} \square \cdots \square P_{n_d}$.

First, observe that every $2$-dimensional grid  has a spanning linear forest, i.e. a union of disjoint paths), that covers all cycles. Since every linear forest is equitably  $k$-choosable for any $k\ge 2$ (cf. Lemma \ref{lem:equitable-colouring-cycles}) then, using Lemma \ref{lem:covering}, we have the following

\begin{theorem}
\label{thm:grids_two}
Let $k\in \mathbb{N}$. If $k\geq 2$ then every $2$-dimensional grid is equitably $k$-list arborable. 
\end{theorem}


\subsection{3-dimensional grids}

\begin{theorem}
\label{thm:grids-three-small}
Let $k,n_2,n_3 \in \mathbb{N}$ with $n_2\geq 2$, $n_3\geq 2$. If $k \geq 2$ then $P_2 \square P_{n_2}\square P_{n_3}$ is equitably $k$-list arborable.
\end{theorem}

\begin{proof}
We will prove that $P_2 \square P_{n_2}\square P_{n_3}$ contains a subgraph $H$ with maximum degree at most two that covers all cycles. Since $P_2 \square P_{n_2}\square P_{n_3}$ is bipartite then $H$ is also bipartite so, by Lemma \ref{lem:equitable-colouring-cycles}, $H$ is certainly equitably  $k$-choosable for any $k\ge 2$. Hence, by Lemma \ref{lem:covering}, the proof will follow.

\begin{figure}[htb]
\begin{center}
\includegraphics[scale=1]{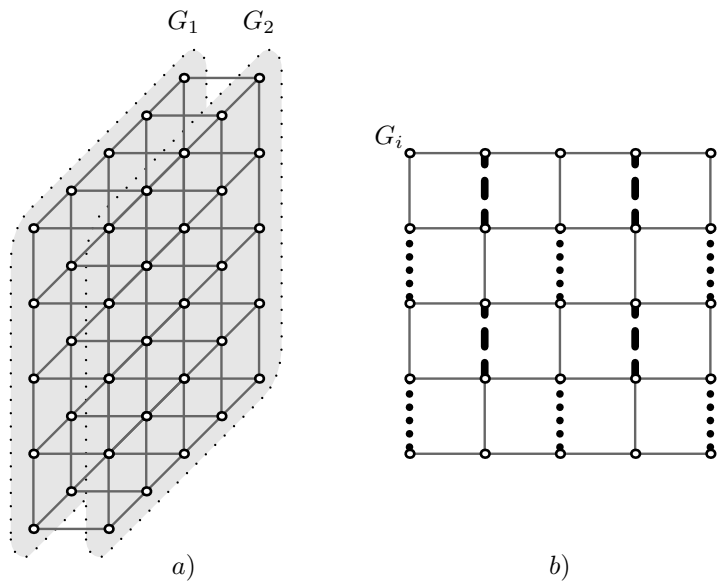}
\caption{Illustration for the proof of Theorem \ref{thm:grids-three-small}; a)$P_2 \square P_{5} \square P_{5}$ with depicted layers $G_1$ and $G_2$; b) layer $G_i$ with depicted set $M_i^\prime$ (dotted line) and set $M_i^{\prime\prime}$ (dashed line).}
\label{two_n_n}
\end{center}
\end{figure}

We can see $P_2 \square P_{n_2}\square P_{n_3}$ as two copies of $P_{n_2}\square P_{n_3}$ (we call them layers $G_1$ and $G_2$) joined by some edges. Let $V(G_1)=\{(1,y,z):y\in [n_2], z\in [n_3]\}$ be the vertex set of the  layer $G_1$ and let $V(G_2)=\{(2,y,z):y\in [n_2],z\in [n_3]\}$ be the vertex set of the  layer $G_2$ (cf. Fig. \ref{two_n_n}a)). In each  layer we choose a maximal matching in the following way. In each column we choose a maximal matching. We start with the first edge  if the column is odd and with the second edge if the column is even. More formally, for $i\in [2]$, $M_i^\prime=\{(i,2p+1,r)(i,2p+2,r):p\in[0, \left\lfloor (n_2-2)/2\right\rfloor],r\in  [n_3],r \;\mbox{is odd}\}$ and $M_i^{\prime\prime}=\{(i,2p,r)(i,2p+1,r):p\in [\left\lfloor (n_2-1)/2\right\rfloor],r\in [n_3],r\; \mbox{is even}\}$ (cf. Fig.~\ref{two_n_n}b)). Let $M_i$ be a spanning subgraph of $G_i$ such that $V(M_i)=V(G_i)$ and $E(M_i)=M_i^\prime\cup M_i^{\prime\prime}$. We show that  $M_i$ covers all cycles in $G_i$. 
Since  both $G_1,G_2$ are isomorphic to $P_{n_2}\square P_{n_3}$ we simplify notation and show that $M=M^\prime\cup M^{\prime\prime}$ covers all cycles in $P_{n_2}\square P_{n_3}$, where $M^\prime=\{(2p+1,r)(2p+2,r):p\in[0, \left\lfloor (n_2-2)/2\right\rfloor],r\in  [n_3],r \;\mbox{is odd}\}$ and $M^{\prime\prime}=\{(2p,r)(2p+1,r):p\in [\left\lfloor (n_2-1)/2\right\rfloor],r\in [n_3],r\; \mbox{is even}\}$. 
We prove it by induction on $n_3$. It is obviously true for $n_3=2$. Thus by induction hypothesis we may assume that such a spanning subgraph covers all cycles of $P_{n_2}\square P_{n_3-1}$. Suppose that $P_{n_2}\square P_{n_3}$ contains a cycle $C$ not covered by $M$. Thus $C$ contains an edge whose vertices have second coordinates $n_3$, say $(x,n_3)(x+1,n_3)$. So  $(x,n_3)(x+1,n_3)\notin M$, however by our choice of $M$ we have $(x-1,n_3)(x,n_3)\in M$ and $(x+1,n_3)(x+2,n_3)\in M$ (whenever such edges exist in $P_{n_2}\square P_{n_3})$.   Thus $C$ must contain vertices $(x,n_3-1),(x+1,n_3-1)$ but $(x,n_3-1)(x+1,n_3-1)\in M$, which contradicts that $M$ does not cover $C$. Now we construct a spanning subgraph $H$ of $P_2 \square P_{n_2}\square P_{n_3}$ in the following way.  Let us denote the set of edges in  $P_2 \square P_{n_2}\square P_{n_3}$ joining vertices between $G_1$ and $G_2$ by $E(G_1,G_2)$. We set $E(H)=M_1\cup M_2\cup E(G_1,G_2)$. Thus $H$ covers all cycles of $P_2 \square P_{n_2}\square P_{n_3}$ and $\Delta (H)=2$, and so $P_2 \square P_{n_2}\square P_{n_3}$ is equitably $k$-list arborable for every $k \geq 2$.
\end{proof}


\begin{theorem}
\label{thm:grids-p_three-p_three-p_n}
Let $n_3, k \in \mathbb{N}$. If $k\geq 2$ then
$P_3 \square P_{3}\square P_{n_3}$ is equitably $k$-list arborable.
\end{theorem}

\begin{proof}
Similarly as in the proof of Theorem \ref{thm:grids-three-small}, we prove that $P_3 \square P_{3}\square P_{n_3}$ contains a spanning subgraph $HP_{3 \times 3\times n_3}$ with maximum degree at most two that covers all cycles. Since $P_3 \square P_{3}\square P_{n_3}$ is bipartite, $HP_{3 \times 3\times n_3}$ is also bipartite so, by Lemma \ref{lem:equitable-colouring-cycles}, $HP_{3 \times 3\times n_3}$ is equitably  $k$-choosable for any $k\ge 2$. Thus, by Lemma \ref{lem:covering}, the proof will follow.

\begin{figure}[htb]
\begin{center}
\includegraphics[scale=1]{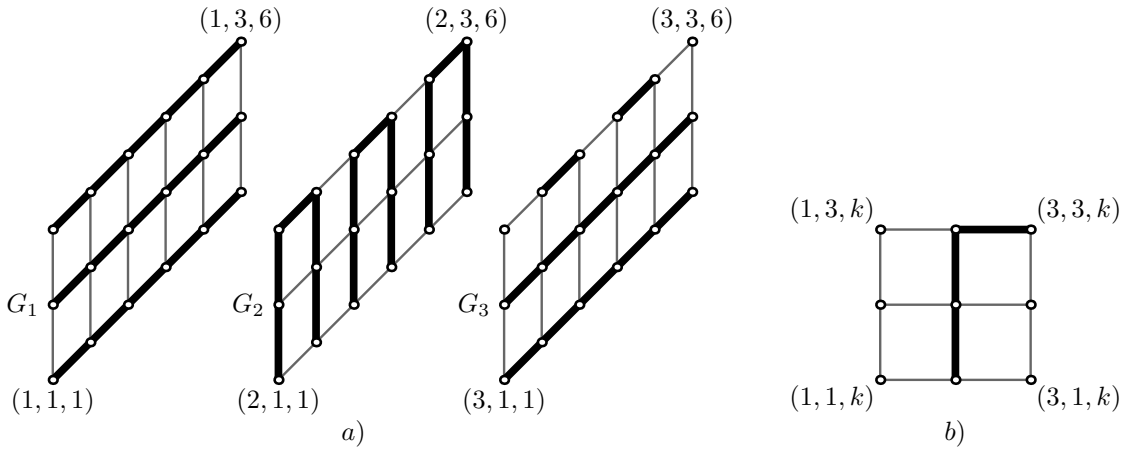}
\caption{Illustration for the proof of Theorem \ref{thm:grids-p_three-p_three-p_n}}
\label{three_three_n}
\end{center}
\end{figure}

Let $G_1,G_2$ and $G_3$ be layers of $P_3 \square P_{3}\square P_{n_3}$ such that $V(G_i)=\{(i,y,z):y\in [3], z\in [n_3]\}$ for $i\in[3]$. In each layer $G_i$ we choose the spanning subgraph $M_i$ in the following way (cf. Fig.~\ref{three_three_n}a)):

\begin{itemize}
\item $E(M_1)=\{(1,i,j)(1,i,j+1): i \in [3], j\in[n_3-1]\}$;
\item $E(M_2)=\{(2,1,i)(2,1,i+1):i\in[n_3-1]_{ODD}\}\cup \{(2,1,i)(2,2,i),(2,2,i)(2,3,i):i\in[n_3]\}$;
\item $E(M_3)=\{(3,1,i)(3,1,i+1):i\in[n_3-1]\}\cup \{(3,2,i)(3,2,i+1):i\in[n_3-1]\}\cup \{(3,3,i)(3,3,i+1):i\in[n_3-1]_{EVEN}\}$.
\end{itemize}

\noindent Moreover, 
\begin{itemize}
    \item $E_{2,3}=\{(2,3,i)(3,3,i):i\in[n_3]\}$.
\end{itemize}

\noindent The subgraph  $HP_{3 \times 3\times n_3}$ is defined in the following way: $V(HP_{3 \times 3\times n_3})=V(P_3 \square P_{3}\square P_{n_3})$ and $E(HP_{3 \times 3\times n_3})=E(M_1)\cup E(M_2)\cup E(M_3)\cup E_{2,3}$.  

We show that $HP_{3 \times 3\times n_3}$ covers all cycles of $P_3 \square P_{3}\square P_{n_3}$. Let $L_i$ for $i\in[n_3]$  be  layers that are isomorphic to $P_3\square P_3$, so $V(L_i)=\{(j,\ell,i):j\in[3],\ell\in[3]\}$. Observe that the subgraphs induced by $V(HP_{3 \times 3\times n_3}) \cap V(L_i)$ are isomorphic (cf. Fig.~\ref{three_three_n}b)). 

If a cycle in $P_{3} \square P_3 \square P_{n_3}$ contains an edge from $HP_{3 \times 3 \times n_3}$ then obviously it is covered by $HP_{3 \times 3 \times n_3}$. Thus we focus only on cycles in $P_{3} \square P_3 \square P_{n_3} - E(HP_{3 \times 3 \times n_3})$. We use the induction method to proof that every cycle in
$P_{3} \square P_3 \square P_{n_3} - E(HP_{3 \times 3 \times n_3})$ contains two vertices $u$ and $v$ such that $uv \in E(HP_{3 \times 3 \times n_3})$. 

It is easy to see that  $HP_{3 \times 3 \times 1}$ covers all cycles in $P_{3} \square P_3 \square P_{1}$.
Let $n_3\geq 2$, assume that $HP_{3 \times 3 \times (n_3-1)}$ covers all cycles in $P_{3} \square P_3 \square P_{n_3-1}$ and consider  $HP_{3 \times 3 \times n_3}$ in  $P_{3} \square P_3 \square P_{n_3}$. Thus if there is an uncovered cycle in $P_{3} \square P_3 \square P_{n_3} - E(HP_{3 \times 3 \times n_3})$ then it must  contain vertices from layer  $L_{n_3}$. First observe that the only cycle of $L_{n_3}$ that contains no edge from $HP_{3 \times 3 \times n_3}$  contains vertices $(2,1,n_3)$ and $(2,2,n_3)$. Since $(2,1,n_3)(2,2,n_3)\in E(HP_{3 \times 3 \times n_3})$, all cycles of $L_{n_3}$ are covered by $HP_{3 \times 3 \times n_3}$. Thus if there is an uncovered cycle $C$ in $P_{3} \square P_3 \square P_{n_3} - E(HP_{3 \times 3 \times n_3})$ then it must  contain vertices from layers $L_{n_3}$ and $L_{n_3-1}$.
We consider two cases.

\medskip 
\noindent {\bf Case 1.} $n_3$ is even. $C$ must go through two out of three following edges: $a=(2,3,n_3-1)(2,3,n_3)$, $b=(2,2,n_3-1)(2,2,n_3)$, $c=(3,3,n_3-1)(3,3,n_3)$.
If $C$ contains edges $a$ and $b$ (edges $a$ and $c$, resp.) then it is covered by the 
edge $(2,2,n_3)(2,3,n_3)$ ($(2,3,n_3)(3,3,n_3)$, resp.).
If $C$ goes through the edges $b$ and $c$ then it 
must contain the vertex $(3,2,n_3)$. On the other hand,  edges $(3,3,n_3-2)(3,3,n_3-1)$ and $(2,3,n_3-1)(3,3,n_3-1)$ belong to  $HP_{3 \times 3 \times n_3}$. Hence $C$ must go through $(3,2,n_3-1)(3,3,n_3-1)$. This implies that the cycle is covered by the edge $(3,2,n_3-1)(3,2,n_3)$.

\medskip
\noindent {\bf Case 2.} $n_3$ is odd.  $C$ must go through two out of three following edges: $a=(2,3,n_3-1)(2,3,n_3)$, $b=(2,2,n_3-1)(2,2,n_3)$, $c=(2,1,n_3-1)(2,1,n_3)$.
If $C$ contains edges $a$ and $b$ ($b$ and $c$, resp.) then it is covered by the
edge $(2,2,n_3)(2,3,n_3)$ ($(2,1,n_3)(2,2,n_3)$, resp.).
If the cycle contains the edges $a$ and $c$ then, to avoid vertex $(2,2,n_3)$, it consecutively goes through the edge $a$, vertices $(1,3,n_3)$, $(1,2,n_3)$, $(1,1,n_3)$, $(2,1,n_3)$ and edge $c$. Observe that $(2,3,n_3-1)$ is incident with exactly two edges $(1,3,n_3-1)(2,3,n_3-1)$ and $(2,3,n_3-2)(1,3,n_3-1)$ that are not in $E(HP_{3 \times 3 \times n_3})$. Due to '$n_3$ even' case the cycle $C$ cannot go through the second one. If it goes through the first one then $(1,3,n_3-1)\in V(C)$ and $C$ is covered by $(1,3,n_3-1)(1,3,n_3)$.

Thus $HP_{3 \times 3 \times n_3}$ covers all cycles of $P_2 \square P_{n_2}\square P_{n_3}$. $\Delta (HP_{3 \times 3 \times n_3})=2$, and so $P_3 \square P_{n_3}\square P_{n_3}$ is equitably $k$-list arborable for every $k \geq 2$.
\end{proof}

\begin{theorem}
\label{main3}
Let $n_1,n_2,n_3,k\in \mathbb{N}$. If $k\geq 3$  then $P_{n_1}\square P_{n_2}\square P_{n_3}$ is equitably $k$-list arborable. 
\end{theorem}

\begin{proof} 
Let $G=P_{n_1}\square P_{n_2}\square P_{n_3}$ be  a $3$-dimensional grid. Let us define a set of edges $X_{ij}=\{(\ell,i,j)(\ell+1,i,j):\ell\in[n_1-1]\}$ for $i\in[n_2]$ and $j\in[n_3]$. First, observe that the graph $(V(G),X)$, where $X=\bigcup_{i\in[n_2],j\in[n_3]}X_{ij}$, is a linear forest. Thus $G-X$ covers all cycles of $G$. Furthermore, every component of $G-X$ is isomorphic to  $P_{n_2}\square P_{n_3}$. Thus, by Lemma \ref{lem:equitable-colouring-p_n_one-p_n_two-final}, $G-X$ is equitably  $k$-choosable for every $k\ge 3$. Finally, Lemma \ref{lem:covering} implies that $G$ is equitably $k$-list arborable for every $k \geq 3$.
\end{proof}


\subsection{4-dimensional grids}

\begin{theorem}
\label{four_dimensional_two}
Let $n_4, k \in \mathbb{N}$. If $k\geq 2$ then $P_{2} \square P_{2} \square P_{2} \square P_{n_4}$ is equitably $k$-list arborable. 
\end{theorem}

\begin{figure}[htb]
\begin{center}
\includegraphics[scale=1]{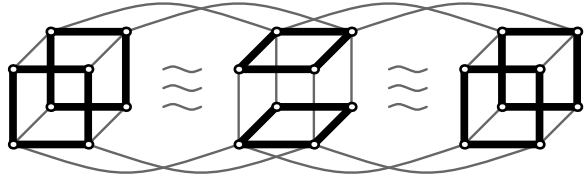}
\caption{Illustration for the proof of Theorem \ref{four_dimensional_two}}
\label{four_dimensional_two-arboricity}
\end{center}
\end{figure}

\begin{proof}
Let $G=P_{2} \square P_{2} \square P_{2} \square P_{n_4}$. We can see $G$ as $n_4$ $3$-dimensional cubes $Q^1,\ldots ,Q^{n_4}$ joined by some edges. Let $H$ be a spanning subgraph of $G$ that contains two cycles of length $4$ of each cube $Q^i$: 'front' and 'back' cycles of $Q^i$ with $i$ odd, 'top' and 'bottom' cycles of $Q^i$ with $i$ even (cf. Fig. \ref{four_dimensional_two-arboricity}). More formally, let us define a spanning subgraph $H$ of $G$   in the following way $E(H)=E_1\cup E_2$, where

\begin{itemize}
\item[$E_1=$]$\{(1,1,1,i)(1,2,1,i),(1,1,1,i)(2,1,1,i),(1,2,1,i)(2,2,1,i),(2,1,1,i)(2,2,1,i),\\
                (1,1,2,i)(1,2,2,i),(1,1,2,i)(2,1,2,i),(1,2,2,i)(2,2,2,i),(2,1,2,i)(2,2,2,i):i\in [n_4]_{ODD}\}$

\item[$E_2=$]$\{(1,1,1,j)(1,1,2,j),(1,1,1,j)(2,1,1,j),(2,1,1,j)(2,1,2,j),(1,1,2,j)(2,1,2,j),\\
       (1,2,1,j)(1,2,2,j),(1,2,1,j)(2,2,1,j),(2,2,1,j)(2,2,2,j),(1,2,2,j)(2,2,2,j):j\in [n_4]_{EVEN}\}$.
 
\end{itemize}

We prove by induction on $n_4$ that $H$ covers all cycles of $P_{2} \square P_{2} \square P_{2} \square P_{n_4}$. It is obviously true for $n_4=1$. Assume that it is true for $P_{2} \square P_{2} \square P_{2} \square P_{n_4-1}$. Without loss of generality we may assume that $n_4$ is even. Suppose that  there is a cycle $C$ in $G$ that has no two vertices adjacent by an edge in $H$. Since  there is no such a cycle in $P_{2} \square P_{2} \square P_{2} \square P_{n_4-1}$, it follows that $C$ contains an edge of the cube $Q^{n_4}$ induced by the vertices of the form $(i,j,\ell, n_4)$, $i\in [2],j\in [2],\ell\in [2]$  that is not in $H$. By symmetry we may assume that $C$ contains $(1,1,1,n_4)(1,2,1,n_4)$. Thus $C$ must also contain vertices  $(1,1,1,n_4-1)$ and $(1,2,1,n_4-1)$, however $(1,1,1,n_4-1)(1,2,1,n_4-1)\in E(H)$, a contradiction. Since $H$ is equitably  $k$-choosable for $k\ge 2$ by Lemma \ref{lem:equitable-colouring-cycles}, $G$  is equitably $k$-list arborable for $k\ge 2$ by Lemma \ref{lem:covering}.
\end{proof}

\begin{theorem}
\label{four_dimensional_three}
Let  $n_3,n_4,k \in \mathbb{N}$. If $k\geq 3$ then $P_{2} \square P_{2} \square P_{n_3} \square P_{n_4}$ is equitably $k$-list arborable. 
\end{theorem}

\begin{figure}[htb]
\begin{center}
\includegraphics[scale=1]{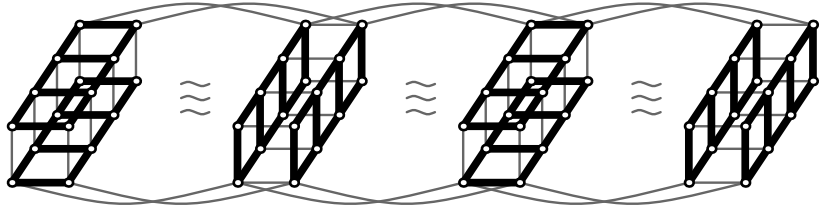}
\caption{Illustration for the proof of Theorem \ref{four_dimensional_three}}
\label{four_dimensional_three-arboricity}
\end{center}
\end{figure}

\begin{proof}
Let $G=P_{2} \square P_{2} \square P_{n_3} \square P_{n_4}$. We show that there is a spanning subgraph $H$ of $G$ that covers all cycles of $G$ such that each component of $H$ is isomorphic to $P_2\square P_{n_3}$. Since  $H$ is equitably  $k$-choosable for $k\ge 3$, by Lemma \ref{lem:equitable-colouring-p_n_one-p_two}, we apply Lemma \ref{lem:covering} to show that $G$ is equitably $k$-list arborable for every $k \geq 3$. We can cf. $G$ as $n_4$  layers $G_1,\ldots,G_{n_4}$, each of which is isomorphic to a $3$-dimensional grid  $P_{2} \square P_{2} \square P_{n_3}$, joined by some edges. To obtain $H$ from every grid $G_i$ we take two disjoint $P_{2} \square P_{n_3}$, if $i$ is odd we take 'top' and 'bottom' $P_{2} \square P_{n_3}$, if $i$ is even we take 'left' and 'right' $P_{2} \square P_{n_3}$ (cf. Fig.  \ref{four_dimensional_three-arboricity}).  Let $H=\bigcup _{i\in [n_4]_{ODD}}(H_{1i}\cup H_{2i})\cup \bigcup_{j\in [n_4]_{EVEN}}(H'_{1j}\cup H'_{2j})$ be a spanning subgraph of $G$, where

\begin{itemize}
\item $H_{1i}=G[\{(1,1,p,i),(2,1,p,i):p\in [n_3]]$ ('bottom'), 
\item $H_{2i}=G[\{(1,2,p,i),(2,2,p,i):p\in n_3]]$ ('top'),
\item $H'_{1j}=G[\{(1,1,p,j),(1,2,p,j):p\in [n_3]]$ ('left'), 
\item $H'_{2j}=G[\{(2,1,p,j),(2,2,p,j):p\in [n_3]]$ ('right').
\end{itemize}

We prove by induction on $n_4$ that $H$ covers all cycles of $G$. It is easy to see that if $n_4=1$, the subgraph $H$ covers all cycles of $G$.  Now, suppose that $H$ covers all cycles of   $P_{2} \square P_{2} \square P_{n_3} \square P_{n_4-1}$. Without  loss of generality we may assume that $n_4$ is odd. If $G$ contains a cycle $C$ not covered by $H$ then there is  an edge in $C$ whose end vertices have the last coordinate $n_4$ and that are not in $H$. Let $(1,1,p,n_4)(1,2,p,n_4)$ be such an edge. Since all edges adjacent to the edge $(1,1,p,n_4)(1,2,p,n_4)$ except $(1,1,p,n_4)(1,1,p,n_4-1)$ and $(1,2,p,n_4)(1,2,p,n_4-1)$ are in $H$ then the vertices $(1,1,p,n_4-1)$ and $(1,2,p,n_4-1)$ must be in $C$.  However, $(1,1,p,n_4-1)(1,2,p,n_4-1)\in E(H)$, which contradicts the assumption that $H$ does not cover $C$. Thus, by Lemma \ref{lem:equitable-colouring-p_n_one-p_two} and  Lemma \ref{lem:covering}, the theorem holds.
\end{proof}

\begin{theorem}
\label{four_dimensional_four}
Every  $4$-dimensional grid  is equitably $4$-list  arborable. 
\end{theorem}
\begin{proof}
Let $G=P_{n_1} \square P_{n_2} \square P_{n_3} \square P_{n_4}$. Again, we  determine a graph $H$, whose every component is isomorphic to $P_2 \square P_{n_2} \square P_{n_3}$ or $P_{n_2}\square P_{n_3}$, that covers all cycles of $G$. Next we  apply Lemmas \ref{lem:equitable-colouring-p_n_one-p_n_two-p_two} and \ref{lem:covering}, so $G$ is equitably $4$-list  arborable. 

\begin{figure}[htb]
\begin{center}
\includegraphics[scale=1]{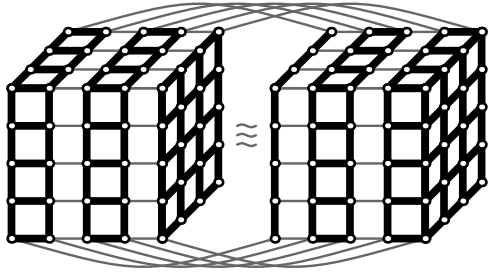}
\caption{Illustration for the proof of Theorem \ref{four_dimensional_four}}
\label{four_dimensional_four-arboricity}
\end{center}
\end{figure}

We can see $G$ as  $3$-dimensional grids  $G_i=P_{n_1} \square P_{n_2} \square P_{n_3},\;i\in[n_4]$ joined by some edges, i.e. $G_i=G[\{(r,s,t,i):r\in[n_1],s\in[n_2],t\in[n_3]\}],;i\in[n_4]$.
To obtain $H$ we take all copies of $G_i$ after removing the matching $E_i$ defined as follows (cf. Fig. \ref{four_dimensional_four-arboricity}). 

\noindent $E_i=\left\{\begin{matrix}
\{(r,s,t,i)(r+1,s,t,i):r\in [n_1-1]_{ODD},s\in[n_2],t\in[n_3]\}& \mbox{if} \;i \;\mbox{is odd,}\\
\{(r,s,t,i)(r+1,s,t,i):r\in [n_1-1]_{EVEN},s\in[n_2],t\in[n_3]\}& \mbox{if}\; i\; \mbox{is even.}\\
\end{matrix}\right.$

\noindent Now, $H=\bigcup_{i\in[n_4]}(G_i-E_i)$. We prove by induction on $n_4$ that $H$ covers all cycles of $G$. Since $E_1$ is a matching, $G_1-E_1$ obviously covers all cycles of $G_1$. Let $G'=P_{n_1} \square P_{n_2} \square P_{n_3} \square P_{n_4-1}$ and $H'=\bigcup_{i\in[n_4-1]}(G_i-E_i)$. Assume that $H'$ covers all cycles of $G'$. Without of loss generality we may assume that $n_4$ is odd. On the contrary, suppose that $G$ contains a cycle $C$ not covered by $H$. Thus $C$ contains an edge $e$ of $E_{n_4}$, say $e=(2r+1,s,t,n_4)(2r+2,s,t,n_4)$. So  vertices $(2r+1,s,t,n_4),(2r+2,s,t,n_4)$ are in $V(C)$. Since all edges of $G_{n_4}$ incident with $(2r+1,s,t,n_4)$ and $(2r+2,s,t,n_4)$, except $e$, are in $H$, we must have that $(2r+1,s,t,n_4-1)$ is a neighbour of $(2r+1,s,t,n_4)$ in $C$ and $(2r+2,s,t,n_4-1)$ is a neighbour of $(2r+2,s,t,n_4)$ in $C$. Thus $(2r+1,s,t,n_4-1),(2r+2,s,t,n_4-1)\in V(C)$, however $(2r+1,s,t,n_4-1)(2r+2,s,t,n_4-1)\in E(H)$, which contradicts that $C$ is not covered by $H$.
\end{proof}

In the proof of the next theorem we use Lemma \ref{lem:vertex_partition_arboricity}. We determine a special $5$-partition of a graph to show that the graph is equitably $k$-list  arborable for every $k \geq 5$. 

\begin{theorem}
Let  $k\in \mathbb{N}$. If $k\geq 5$ then every $4$-dimensional grid is equitably $k$-list  arborable. 
\label{main4}
\end{theorem}
\begin{proof}
Let $G=P_{n_1} \square P_{n_2} \square P_{n_3} \square P_{n_4}$ and $V(G)=\{(i,j,k,l): i \in [n_1], j \in [n_2], k \in [n_3], l \in [n_4]\}$,
We  determine a special $5$-partition $S_1 \cup \cdots \cup S_{\eta+1}$ of $G$, with $|V(G)|=5\eta+r$ and $r\in [5]$, that fulfills the assumptions of Lemma \ref{lem:vertex_partition_arboricity}.  So, by Lemma \ref{lem:vertex_partition_arboricity}, the theorem will follow. 
We  depict sets $S_j$ of size $5$ step by step in decreasing order, starting with determining a set $S_{\eta+1}$ and next, in the same manner,  sets $S_{\eta}, \ldots ,S_2$. The last set $S_1$ is formed by vertices in $V(G) \backslash (S_2 \cup \cdots \cup S_{\eta+1})$, so its size   is less than or equal to $5$. Since the assumtions of Lemma \ref{lem:vertex_partition_arboricity} are obviously fulfilled for each $4$-dimentional grid $G$ satisfying $|V(G)|\le 5$ we may assume that $|V(G)|\geq 6$.

Let $j\in [2,\eta+1]$. To determine a set $S_j$ consisting of elements $x_1^j, \ldots ,x_5^j$, we use the sets $S_{j+1}, \ldots ,S_{\eta+1}$ constituted in the previous steps. 
Let $G_j=G-(S_{j+1}\cup \cdots \cup S_{\eta+1})$. Thus $G_j$ is the graph induced in $G$ by the union of sets $S_1, \ldots ,S_j$, whose forms are unknown at this moment. Observe that $V(G_{j-1})$ is equal to $S_1\cup \cdots \cup S_{j-1}$. Hence $V(G_{j-1})$ is the set involved in the condition (\ref{ineq_two}) of Lemma  \ref{lem:vertex_partition_arboricity}. Precisely, this condition  can be rewritten here in the form
$$|N_{G_{j-1}}(x_i^j)|\le 2i-1.$$

To find $x_1^j, \ldots ,x_5^j$ that satisfy the condition (\ref{ineq_two}) of Lemma \ref{lem:vertex_partition_arboricity}, let us  do as follows.

Let $L_{lex}$ be the list of all vertices of $V(G_{j-1})$ ordered lexicographically. Note that if vertex $(a,b,c,d)$ is the first in the list then it has at most four neighbours in the list: $(a+1,b,c,d)$, $(a,b+1,c,d)$, $(a,b,c+1,d)$, $(a,b,c,d+1)$, moreover if it has exactly four neighbours then $(a,b,c,d+1)$ is the second in the list.

Let $x_1^j$ be the first, $x_2^j$ the second and $x_3^j$ the third vertex in the list $L_{lex}$. Remove those vertices from the list. If there is still any neighbour of $x_1^j$ in the list then let $x_4^j$ be this neighbour, otherwise let $x_4^j$ be the first element in the list. Remove $x_4^j$ from the list and similarly choose $x_5^j$. If there is any neighbour of $x_1^j$ in the list then let $x_5^j$ be this neighbour, otherwise let $x_5^j$ be the first element in the list.

We will prove that the set $S_j$, determined in the way described above, fulfill the assumption of Lemma \ref{lem:vertex_partition_arboricity}.
We know that $|N_{G_{j}}(x_1^j)|\le 4$. If $|N_{G_{j}}(x_i^j)| = 4$ then we have chosen to $S_j$ at least three of the neighbours of $x_1^j$: $x_2^j$, $x_4^j$, $x_5^j$. On the other hand, if $2 \leq |N_{G_{j}}(x_1^j)|\leq 3$ then at least two neighbours of $x_1^j$ are chosen to $S_j$. In every case we have $|N_{G_{j-1}}(x_1^j)|\le 1$. If $|N_{G_{j}}(x_2^j)| = 4$ then $x_2^j$ and $x_3^j$ are adjacent, so $|N_{G_{j-1}}(x_2^j)|\le 3$. After removing $x_1^j$ and $x_2^j$, the vertex $x_3^j$ was the first in the list so $|N_{G_{j-1}}(x_3^j)|\le 4 \leq 5$. If $x_4^j$ was chosen as the first in the list then $|N_{G_{j-1}}(x_4^j)|\le 4$, otherwise at least one of its neighbours, i.e. $x_1^j$, is in $S_j$, so $|N_{G_{j-1}}(x_4^j)|\le 7$. Obviously $|N_{G_{j-1}}(x_5^j)|\le 9$.

\end{proof}


\subsection{$d$-dimensional grids, the general upper bound}

In Section \ref{lem:generalization} we give a general upper bound on the equitable list vertex arboricity of all graphs. Now we improve this bound for $d$-dimensional grids.

Assume that $d\geq 3$ and $n_1, \ldots ,n_{d-2}\in \mathbb{N}\setminus \{1\}$. Let us define the following family of graphs.

\medskip 
\noindent ${\cal H}(n_1,\ldots,n_{d-2})= \{G: \mbox{each component of}\; G \;\mbox{is isomorphic to}\; P_{n_1} \square \cdots \square P_{n_{d-2}}\square P_2 \\ \mbox{or}\; P_{n_1} \square \cdots \square P_{n_{d-2}}\}$.

\begin{lemma}
\label{lem:arboricity_general_bound}
Let $d\in \mathbb{N}$ with $d \geq 3$, $n_1, \ldots ,n_{d-2}\in \mathbb{N}\setminus \{1\}$ and $G=P_{n_1} \square \ldots \square P_{n_d}$. There is a graph $H\in {\cal H}(n_1,\ldots,n_{d-2})$ that covers all cycles of $G$.
\end{lemma}
\begin{proof}
The idea of determining a graph $H$ is the same as in the proof of Theorem \ref{four_dimensional_four}.
We can see $G$ as $n_d$ copies of a $(d-1)$-dimensional grid  $P_{n_1} \square \cdots \square P_{n_{d-1}}$ joined by some edges. Let $G_i=G[\{(y_1,\ldots,y_{d-1},i):y_j\in[n_j],j\in[d-1]\}]$, $i\in [n_d]$. To obtain $H$, we delete from every $G_i$ the matching $E_i$ defined as follows.


\begin{description}
    \item[Case 1] $i$ is odd
    $$E_i=\{(y_1,y_2,\ldots,y_{d-1},i)(y_1+1,y_2,\ldots,y_{d-1},i):\\y_1\in [n_1-1]_{ODD}\}.$$
    \item[Case 2] $i$ is even
    $$E_i=\{(y_1,y_2,\ldots,y_{d-1},i)(y_1+1,y_2,\ldots,y_{d-1},i):\\y_1\in [n_1-1]_{EVEN}\}.$$
\end{description}

\noindent In both cases we take $y_j\in[n_j]$ for $j\in[2,d-1]$.
\medskip
\noindent Put $H=\bigcup_{i\in[n_d]}(G_i-E_i)$. Note that $H\in {\cal H}(n_1,\ldots,n_{d-2})$. We prove by induction on $n_d$ that $H$ covers all cycles of $G$. Since $E_1$ is a matching of $G_1$, obviously $G_1-E_1$  covers all cycles of $G_1$. Let $G'=P_{n_1} \square \cdots \square P_{n_{d}-1}$ and $H'=\bigcup_{i\in[n_{d}-1]}(G_i-E_i)$. By the induction hypothesis, $H'$ covers all cycles of $G'$. Without  loss of generality we may assume that $n_d$ is odd. On the contrary, suppose that $G$ contains a cycle $C$ not covered by $H$. Thus $C$ contains an edge $e$ of $E_{n_d}$, say $e=(2r+1,y_2,\ldots,y_{d-1},n_d)(2r+2,y_2,\ldots,y_{d-1},n_d)$. So  vertices $(2r+1,y_2,\ldots,y_{d-1},n_d),(2r+2,y_2,\ldots,y_{d-1},n_d)$ are in $V(C)$. Since all edges of $G_{n_d}$ incident with $(2r+1,y_2,\ldots,y_{d-1},n_d)$ and $(2r+2,y_2,\ldots,y_{d-1},n_d)$, except $e$, are in $H$, we must have that $(2r+1,y_2,\ldots,y_{d-1},n_d-1)$ is a neighbour of $(2r+1,y_2,\ldots,y_{d-1},n_d)$ in $C$ and $(2r+2,y_2,\ldots,y_{d-1},n_d-1)$ is a neighbour of $(2r+2,y_2,\ldots,y_{d-1},n_d)$ in $C$. Thus $(2r+1,y_2,\ldots,y_{d-1},n_d-1),(2r+2,y_2,\ldots,y_{d-1},n_d-1)\in V(C)$, however $(2r+1,y_2,\ldots,y_{d-1},n_d-1)(2r+2,y_2,\ldots,y_{d-1},n_d-1)\in E(H)$, which contradicts that $C$ is not covered by $H$.
\end{proof}

\begin{observation}
\label{obs:Delta}
Let $d\in \mathbb{N}$, $n_1, \ldots ,n_{d-2}\in \mathbb{N}\setminus \{1\}$ and $H\in {\cal H}(n_1,\ldots,n_{d-2})$. If $d \geq 3$ then $\Delta(H)\le 2d-3$.
\end{observation}

Observation \ref{obs:Delta} together with Theorem \ref{thm:Kierstead-Kostochka}(i)-(ii) and Lemma \ref{lem:covering} imply the following result.

\begin{theorem} Let $d,k\in \mathbb{N}$.
\begin{enumerate}[(i)]
\item If   $k\ge 8$ then every $5$-dimensional grid is equitably $k$-list arborable.
\item If $d\in [6,16]$ and $k\ge 2d-2+\frac{2d-4}{7}$ then every $d$-dimensional grid is equitably $k$-list arborable.
\item If $d\ge 17$ and $k\ge 2d-3+\frac{2d-3}{6}$ then every $d$-dimensional grid is equitably $k$-list arborable.
\end{enumerate}
\end{theorem}

\section{Concluding remarks}
\label{conclusion}

Note that   our results  confirm Zhang's conjectures  for $d$-dimensional grids, when $d \in [2,4]$. For many cases they are even stronger than the conjectures. More precisely, we have obtained the following facts.

\begin{corollary}
Let $k \in \mathbb{N}$ and $d\in \{2,3,4\}$. If $G$ is a $d$-dimensional grid and $k\ge \left\lceil (\Delta(G)+1)/2\right\rceil$ then $G$ is equitably $k$-list arborable.
\end{corollary}

\begin{corollary}
Let  $d,k\in \mathbb{N}$ with $d\geq 2$ and $k\geq 2$. If $G$ is a $d$-dimensional grid with $\Delta(G)\le 5$ then $G$ is equitably $k$-list arborable.
\end{corollary}

\begin{corollary}
Let $k \in \mathbb{N}$, $d\in \{2,3,4\}$, and let $G$ be a $d$-dimensional grid with $\Delta(G)\geq 6$ that is different from $P_{n_1}\square P_{n_2}\square P_{n_3}\square P_2$, $n_1,n_2,n_3\in \mathbb{N}\setminus \{1,2\}$.
If $k\ge \left\lfloor (\Delta(G))/2\right\rfloor$ then $G$  is equitably $k$-list arborable.
\end{corollary}




Since  $d$-dimensional grids have many special properties, we expect that the results that are better than Zhang's conjectures hold for almost all  of them. Among others, $d$-dimensional grids are bipartite and $d$-degenerate. The equitable colouring of such classes of graphs is analyzed in many papers. For instance, it was proven in  \cite{LiWu96} that the inequality $\chi^=(G)\le \Delta(G)$ holds for every connected bipartite graph $G$.  We improve this result for all $d$-dimensional grids. The following two theorems will help us to post some conjectures.

\begin{theorem}
\label{bounded}
Let $d,k\in \mathbb{N}$ with $d \geq 2$, and let $G$ be a $d$-dimensional grid. If  $k \geq 2$ then there exists an equitable proper $k$-colouring  of $G$.  
\end{theorem}
\noindent

The concept of layers in $d$-dimensional grids, used until now, must be extended on the purpose of the proof of Theorem \ref{bounded}. Let $G=P_{n_1} \square \cdots \square P_{n_d}$ and $\{i_1,\ldots,i_s\}$ be any  $s$-subset of indexes from $[d]$. Moreover, let 
$(a_{i_1}, \ldots ,a_{i_s})$ be a fixed $s$-tuple from $[n_{i_1}]\times \cdots \times [n_{i_s}]$. Then each graph induced in $G$ by the set 
$$\{(y_1, \ldots ,y_d):\; y_{i_1}=a_{i_1}, \ldots ,y_{i_s}=a_{i_s}\}$$
 is called an \emph{$s$-layer} of $G$. Note that the layers used until now are $1$-layers.

\begin{prooftw}{\ref{bounded}}
Let   $k$ be fixed and $G=P_{n_1} \square \cdots \square P_{n_d}$ with $n_1, \ldots ,n_d\in \mathbb{N}\setminus \{1\}$. We construct a proper $k$-colouring of $G$ in which every colour class has the cardinality either $\left\lceil |V(G)|/k \right\rceil$ or $\left\lfloor |V(G)|/k \right\rfloor$. The construction is given in $d$ stages. For $i\in [d]$, in the $i-th$ stage we describe a proper $k$-colouring $c_i$ of an $i$-dimensional grid $P_{n_1} \square \cdots \square P_{n_i}$ which is a $(d-i)$-layer $G_i$  of $G$ induced in $G$ by the set of vertices $V_i$, where 
$$V_i=\{(y_1, \ldots ,y_i,\underbrace{1,\ldots ,1}_{d-i}):\; y_1\in [n_1], \ldots ,y_i\in [n_i]\}.$$
We construct a proper $k$-coloring $c_{i+1}$ on $V_{i+1}$ as an extention of a proper $k$-colouring $c_i$ on $V_i$. Finally, we obtain a proper $k$-colouring $c_d$ of $G$. For each $i\in [d]$ we care for $c_i$ to be equitable, which means that each colour class of $c_i$ is of the cardinality either $\left\lceil (n_1\cdots n_i)/k \right\rceil$ or $\left\lfloor (n_1\cdots n_i)/k \right\rfloor$.

Let us start with the construction of $c_1$. In this case $G_1=P_{n_1}$ and we put $c_1((y_1,\underbrace{1,\ldots ,1}_{d-1}))\equiv y_1(\bmod\; k).$ Thus, depending on $n_1$, each of $k$ colours arises either  $\left\lceil n_1/k \right\rceil$ or $\left\lfloor n_1/k \right\rfloor$  times and moreover, $c_1$ is a proper $k$-colouring of $G_1$. Note that this time we use colors from $[0,k-1]$.

Suppose that, for some $i\in [d-1]$, the colouring $c_i$ is constructed. Of course $c_i$ satisfies all requirements mentioned before. Now we permute coloures used in $c_i$ on vertices in $V_i$ (recall that $|V_i|=n_1\cdots n_i$) in such a way that each of the coloures $1, \ldots ,p$ is used $\left\lceil (n_1\cdots n_i)/k\right\rceil$ times and each of the remaining $k-p$ coloures $p+1, \ldots ,k$ is used $\left\lfloor (n_1\cdots n_i)/k \right\rfloor$ times. Of course it could be $p=k$. Now let us define  $c_{i+1}$ for each tuple  $(y_1, \ldots ,y_{i+1})\in [n_1]\times \cdots \times [n_{i+1}]$.  We put\\   
$c_{i+1}((y_1,\ldots y_{i},y_{i+1},\underbrace{1,\ldots ,1}_{d-i-1}))=\left\{\begin{matrix}
(c_{i}((y_1,\ldots y_{i},\underbrace{1,\ldots ,1}_{d-i}))+p(y_{i+1}-1))(\bmod \; k), &\mbox{if}\; p\neq k\\
(c_{i}((y_1,\ldots y_{i}\underbrace{1,\ldots ,1}_{d-i}))+y_{i+1}-1)(\bmod \; k), & \mbox{if}\; p=k\\
\end{matrix}\right.$.

Note that $c_{i+1}$ is proper. Indeed, the graph induced in $G_{i+1}$ by vertices with  fixed coordinate $y_{i+1}$ is isomorphic to $G_i$ and is coloured according  to $c_i$ (with permuted coloures). Moreover, each edge $e$ of $G_{i+1}$ that is not an edge of any copy of $G_i$ (any of the $n_{i+1}$ layers of $G_{i+1}$ that are isomorphic to $G_i$), joins vertices from the consecutive copies of $G_i$ that are consecutive layers of $G_{i+1}$. Hence $e$ has end vertices coloured with  $j$ and $(j+p)(\bmod\; k)$, when $p\neq k$ and $j$ and $(j+1)(\bmod \; k)$, when $k=p$ (for some $j\in [k]$). In both cases these two coloures are different. Thus $c_{i+1}$ is proper.

Next we have to observe that $c_{i+1}$ is equitable. Suppose that $p=k$. In this case each of $k$ coloures arises in $c_i$ on the same number of vertices  in $V_i$. Since in $G_{i+1}$ each of $n_{i+1}$ copies of $G_i$ is coloured in the same manner (with permuted coloures) we can see that in the whole graph $G_{i+1}$ each colour arises the same number $(n_1\cdots n_{i+1})/k$ of times. Consequently $c_{i+1}$ is equitable in this case. Now, suppose that $p\neq k$. Recall that the vertices of the first layer of $G_{i+1}$ are coloured  in such a way that coloures $1, \ldots ,p$ arise one more than coloures $p+1, \ldots k$. In the second layer the coloures $(p+1)(\bmod \; k), \ldots ,(p+p)(\bmod \; k)$ arise one more than the remaining $k-p$ coloures $(p+p+1)(\bmod \; k), \ldots (p+p+k-p)(\bmod \; k)$ and so on. Thus we use coloures cyclically, which guarantees that $c_{i+1}$ is equitable also in this case.
\end{prooftw}

 It is very easy to observe the following fact valid for all $d$-degenerate graphs.

\begin{theorem}
\label{arborable}
Let $d,k\in \mathbb{N}$. If $k \geq \left\lceil (d+1)/2 \right\rceil$ then every $d$-degenerate graph is $k$-list arborable. 
\end{theorem}
\begin{proof}
Let $k$ be fixed. We order vertices  $v_i, \ldots ,v_n$ of $G$ such  that $\deg_{G[\{v_1, \ldots ,v_i\}]}(v_i)\le d$. Such an ordering always exists since $G$ is $d$-degenerate. Let $L$ be an arbitrary $k$-uniform list assignment for $G$. We construct an $L$-colouring of $G$ whose each colour class induces an acyclic subgraph of $G$. We do it, step by step, putting on a vertex $v_i$ a colour from its list that is not present more than once on previously coloured vertices $v_1, \ldots ,v_{i-1}$. Since the size of each list is at least $\left\lceil (d+1)/2 \right\rceil$, such a colour  exists. Obviously, we obtained an $L$-colouring for $G$. Moreover, putting the colour on $v_i$ we do not produce any monochromatic cycle since $v_i$ has at most one neighbour in the colour of $v_i$.
\end{proof}

As we mentioned previously,  a $d$-dimensional grid is $d$-degenerate graph and hence it is $k$-list  arborable for every $k \geq \left\lceil (d+1)/2\right\rceil$, by Theorem \ref{arborable}. Furthermore, when $k\neq 1$, by Theorem \ref{bounded}, for a $d$-dimensional grid there is a $k$-colouring, in which, each colour class is of the cardinality at most $\left\lceil |V(G)|/k\right\rceil$ and induces an acyclic graph (each edgless graph is acyclic). These two facts and some other investigation yield the proposition of a general conjecture. If the conjecture is true then it improves our results for 3-dimensional and 4-dimensional grids.

\begin{conjecture}
Let $k,d\in \mathbb{N}$. If $k \geq \left\lceil (d+1)/2 \right\rceil$ then every $d$-dimensional grid is equitably $k$-list arborable.
\end{conjecture}

However, we do not think that such a conjecture is true in general, i.e., to be $k$-list  arborable and to have a $k$-colouring in which each colour class is of the cardinality at most $\left\lceil |V(G)|/k\right\rceil$ and induces an acyclic graph, is not the sufficient condition to be equitably $k$-list  arborable. Thus we propose the following conjecture.

\begin{conjecture}
There is  a graph $G$ and  $k\in \mathbb{N}$ such that $G$ is $k$-list  arborable and $G$  has a $k$-colouring in which each colour class is of the cardinality at most $\left\lceil |V(G)|/k\right\rceil$ and induces an acyclic graph, however $G$ is not equitably $k$-list arborable.
\end{conjecture}

Note that the motivation of the paper came from Zhang's conjectures, but along the way, we have obtained some new results on equitable $k$-choosability of grids.


\end{document}